\documentclass{amsart}
\usepackage[english]{babel}
\usepackage{amsmath,amsthm,amssymb,amsfonts}
\usepackage{hyperref}

\numberwithin{equation}{section}

\newtheorem{theo}{Theorem}[section]
\newtheorem{prp}[theo]{Proposition}
\newtheorem{lem}[theo]{Lemma}
\newtheorem{cor}[theo]{Corollary}

\newcommand{\tc}{\,:\,}                             
\newcommand{\defeq}{\stackrel{\mathrm{def}}{=}}

\newcommand{\chr}{\chi}                             
\DeclareMathOperator{\supp}{\mathrm{supp}}          

\newcommand{\R}{\mathbb{R}}                         
\newcommand{\N}{\mathbb{N}}                         
\newcommand{\Z}{\mathbb{Z}}                         
\newcommand{\C}{\mathbb{C}}                         
\newcommand{\T}{\mathbb{T}}                         

\newcommand{\D}{\mathcal{D}}                        
\newcommand{\loc}{\mathrm{loc}}                     

\newcommand{\Diff}{\mathfrak{D}}                    
\newcommand{\dstar}{\circ}                          
\newcommand{\dtwist}{\bullet}                       

\newcommand{\LA}{\mathrm{L}}

\newcommand{\RA}{\mathrm{R}}

\newcommand{\lie}{\mathfrak}                        
\DeclareMathOperator{\Ad}{\mathrm{Ad}}              
\DeclareMathOperator{\ad}{\mathrm{ad}}              
\DeclareMathOperator{\Span}{\mathrm{span}}          

\newcommand{\HS}{\mathrm{HS}}							          

\newcommand{\UEnA}{\mathrm{U}}                      

\newcommand{\Kern}{\mathcal{K}}                     

\newcommand{\HH}{\mathcal{H}}                       

\newcommand{\JJ}{\mathcal{J}}                       

\newcommand{\HP}[2]{{\rm(I$_{{#1},{#2}}$)}}     
\newcommand{\HPK}[2]{{\rm(J$_{{#1},{#2}}$)}}    

\newcommand{\strong}{\mathrm{strong}}               

\renewcommand{\ell}{\varrho}

\begin{document}
\title[Joint spectral multipliers on groups of polynomial growth]{Analysis of joint spectral multipliers\\ on Lie groups of polynomial growth}
\author{Alessio Martini}

\begin{abstract}
We study the problem of $L^p$-boundedness ($1 < p < \infty$) of operators of the form $m(L_1,\dots,L_n)$ for a commuting system of self-adjoint left-invariant differential operators $L_1,\dots,L_n$ on a Lie group $G$ of polynomial growth, which generate an algebra containing a weighted subcoercive operator. In particular, when $G$ is a homogeneous group and $L_1,\dots,L_n$ are homogeneous, we prove analogues of the Mihlin-H\"ormander and Marcinkiewicz multiplier theorems.
\end{abstract}

\maketitle

\section{Introduction}\label{section:introduction}

Let $(X,\mu)$ be a measure space, and let $L_1,\dots,L_n$ be (possibly unbounded) self-adjoint operators on $L^2(X,\mu)$ which commute strongly, i.e., which admit a joint spectral resolution $E$ on $\R^n$. Then a joint functional calculus for $L_1,\dots,L_n$ is defined via spectral integration and, for every Borel function $m : \R^n \to \C$, the operator
\[m(L) = m(L_1,\dots,L_n) = \int_{\R^n} m \,dE\]
is bounded on $L^2(X,\mu)$ if and only if the function $m$ is ($E$-essentially) bounded. The characterization of the boundedness of $m(L)$ on other function spaces, such as $L^p(X,\mu)$ for $p \neq 2$, in terms of properties of the function $m$ --- which will be called \emph{(spectral) multiplier} for $L_1,\dots,L_n$ --- is a much more difficult question, even in particular cases. Several problems and results of harmonic analysis fall into this frame, the classical examples being the Mihlin-H\"ormander and Marcinkiewicz theorems for Fourier multipliers on $\R^n$, which give sufficient conditions for the $L^p$ boundedness ($1 < p < \infty$) of joint functions $m(-i\partial_1,\dots,-i\partial_n)$ of the partial derivatives on $\R^n$, in terms of smoothness properties of the multiplier $m$.

Here we are interested in the case $X = G$ is a (connected) Lie group of polynomial volume growth (with a Haar measure $\mu$) and $L_1,\dots,L_n$ belong to the algebra $\Diff(G)$ of left-invariant differential operators on $G$, with particular reference to homogeneous operators $L_1,\dots,L_n$ on a homogeneous (nilpotent) Lie group $G$. For $n = 1$, i.e., for a single operator $L = L_1$, several results of Mihlin-H\"ormander type are known (see, e.g., \cite{mauceri_vectorvalued_1990,christ_multipliers_1991,hebisch_multiplier_1993,mller_spectral_1994,alexopoulos_spectral_1994,hebisch_multiplier_1995,sikora_norms_1999,cowling_spectral_2001,duong_plancherel-type_2002} and references therein), dealing mainly with a sublaplacian $L$ on $G$, or more generally with a positive operator $L$ for which Gaussian-type heat kernel estimates hold. For $n > 1$, instead, most of the known results \cite{mauceri_zonal_1981,mller_marcinkiewicz_1995,mller_marcinkiewicz_1996,fraser_marcinkiewicz_1997,veneruso_marcinkiewicz_2000,fraser_convolution_2001,fraser_fold_2001} refer to a specific class of groups $G$ (namely, the H-type groups) and to specific choices of the operators $L_1,\dots,L_n$ (i.e., sublaplacians and central derivatives); a recent work of Sikora \cite{sikora_multivariable_2009}, which has no such restrictions (in fact it applies to more general settings than Lie groups), is however limited to the case of direct products ($G = G_1 \times \dots \times G_n$ and each $L_j$ operates on a different factor $G_j$ of $G$), so that it does not cover the mentioned results on H-type groups.

In this paper, we propose a quite general setting where spectral multiplier theorems of Mihlin-H\"ormander and Marcinkiewicz type for a system of operators can be obtained. Starting from the \emph{weighted subcoercive operators}\footnote{The definition of weighted subcoercive operator to which we refer in the present work is the one from \cite{martini_spectral}, which is more restrictive than the original one in \cite{ter_elst_weighted_1998}. The restriction appears to be necessary in order to correct an error in \cite{ter_elst_weighted_1998}, as it is explained in \cite[footnote 1]{martini_spectral}. We are not sure if our results remain true under the wider definition.} of ter Elst and Robinson \cite{ter_elst_weighted_1998} --- which are a large class of left-invariant differential operators on a Lie group (including positive elliptic operators, sublaplacians, and positive Rockland operators on homogeneous groups) for which Gaussian-type heat kernel estimates hold --- we define a \emph{weighted subcoercive system} to be a system $L_1,\dots,L_n$ of pairwise commuting, formally self-adjoint, left-invariant differential operators on a connected Lie group $G$ which generate a subalgebra of $\Diff(G)$ containing a weighted subcoercive operator. An ``abstract'' study of weighted subcoercive systems, in relation with the algebraic structure and the representation theory of the environment Lie group, is performed in \cite{martini_spectral}, from which we get in particular that:
\begin{itemize}
\item the members of a weighted subcoercive system $L_1,\dots,L_n$ on $G$ are essentially self-adjoint and commute strongly (so that they admit a joint spectral resolution) in every unitary representation of $G$;
\item the operators of the form $m(L) = m(L_1,\dots,L_n)$ are convolution operators:
\[m(L) \phi = \phi * \Kern_L m = \phi * \breve m\]
for some distribution $\breve m = \Kern_L m$; in fact, if $m$ is bounded and compactly supported, then $\Kern_L m \in L^2(G)$ together with all its left-invariant derivatives;
\item a Plancherel formula holds:
\[\|\Kern_L m\|_{L^2(G)} = \|m\|_{L^2(\R^n,\sigma)}\]
for some regular Borel measure $\sigma$, which is called the \emph{Plancherel measure} associated with $L_1,\dots,L_n$, and whose support is their joint $L^2$ spectrum;
\item if $G$ is a homogeneous group (with automorphic dilations $\delta_t$), a commuting system $L_1,\dots,L_n$ of \emph{homogeneous}, formally self-adjoint left-invariant differential operators is a weighted subcoercive system if and only if $L_1,\dots,L_n$ are \emph{jointly injective} on the smooth vectors $v$ of every non-trivial irreducible representation $\pi$ of $G$:
\[d\pi(L_1) v = \dots = d\pi(L_n) v = 0 \quad\Longrightarrow\quad v = 0\]
(this is a multi-variate analogue of the Rockland condition \cite{rockland_hypoellipticity_1978,helffer_caracterisation_1979}); in this case, we speak of a \emph{homogeneous weighted subcoercive system} and, if $L_1,\dots,L_n$ are homogeneous of degrees $w_1,\dots,w_n$ respectively, then
\begin{equation}\label{eq:homogeneity}
\sigma(\epsilon_t(A)) = t^{Q_\delta} \sigma(A), \qquad \Kern_L (m \circ \epsilon_t) = t^{-Q_\delta} (\Kern_L m) \circ \delta_{t^{-1}},
\end{equation}
where $Q_\delta$ is the homogeneous dimension ($\det \delta_t = t^{Q_\delta}$) and
\[\epsilon_t(\lambda_1,\dots,\lambda_n) = (t^{w_1} \lambda_1,\dots,t^{w_n}\lambda_n)\]
are the dilations on $\R^n$ associated with $L_1,\dots,L_n$.
\end{itemize}

In the following, under the hypothesis that $L_1,\dots,L_n$ is a weighted subcoercive system on a Lie group $G$ of polynomial growth of degree $Q_G$, we prove weighted $L^1$ estimates for the kernels $\Kern_L m$ corresponding to multipliers $m$ with compact support, in terms of a Sobolev norm of $m$. If we suppose further that $G$ is a homogeneous group and that $L_1,\dots,L_n$ are homogeneous operators (with degrees $w_1,\dots,w_n$ and associated dilations $\epsilon_t$), then a theorem of Mihlin-H\"ormander type can be obtained: the operator $m(L)$ is of weak type $(1,1)$ and bounded on $L^p(G)$ for $1 < p < \infty$ whenever the multiplier $m$ satisfies an $L^q$ \emph{Mihlin-H\"ormander condition} of order $s$, i.e.,
\begin{equation}\label{eq:mhcondition}
\|m\|_{M_\epsilon W_q^s} \defeq \sup_{t > 0} \| (m \circ \epsilon_t) \,\eta \|_{W_q^s} < \infty,
\end{equation}
with $q \in [2,\infty]$ and
\begin{equation}\label{eq:mhthreshold}
s > \frac{Q_G}{2} + \frac{n-1}{q},
\end{equation}
where $W_q^s(\R^n)$ is the $L^q$ Sobolev space of (fractional) order $s$ and $\eta$ is a non-negative smooth cut-off function on $\R^n$ supported on an annulus centered at the origin. Notice that the condition \eqref{eq:mhcondition} is independent on the choice of the cut-off $\eta$; moreover, an $L^\infty$ Mihlin-H\"ormander condition of integral order $s$ is essentially equivalent to the pointwise conditions
\begin{equation}\label{eq:mhpointwise}
\sup_{\lambda \neq 0} |\lambda|_\epsilon^{\|\alpha\|_\epsilon} \, |\partial^\alpha m(\lambda)| < \infty
\end{equation}
for $\alpha \in \N^n$ with $|\alpha| = \alpha_1 + \dots + \alpha_n \leq s$, where $|\cdot|_\epsilon$ is a $\epsilon_t$-homogeneous norm on $\R^n$ and $\|\alpha\|_\epsilon = w_1 \alpha_1 + \dots + w_n \alpha_n$.

In particular cases, the previous theorem can be improved by lowering the \emph{regularity threshold}, i.e., the right-hand side of \eqref{eq:mhthreshold}. For instance, by extending a technique due to Hebisch and Zienkiewicz \cite{hebisch_multiplier_1995}, we prove that, if $G$ is the direct product of Euclidean and M\'etivier groups, then the dimension at infinity $Q_G$ can be replaced in \eqref{eq:mhthreshold} by the topological dimension $\dim G$; in fact, in \S\ref{section:metivier} we propose a systematic approach for exploiting such technique, by introducing the notion of \emph{$h$-capacious groups}, for which $Q_G$ can be replaced in \eqref{eq:mhthreshold} by $Q_G - h$. The term $(n-1)/q$ in \eqref{eq:mhthreshold} can be lowered too, by determining the volume growth rate with respect to the Plancherel measure $\sigma$ of Euclidean balls with small radius: namely, if $\sigma(B(\lambda,r)) \leq Cr^d$ for $|\lambda| = 1$ and $r \leq 1$, then $n-1$ can be replaced in \eqref{eq:mhthreshold} by $n-d$.

Finally, a sort of product theory can be developed, by considering several homogeneous Lie groups $G_l$, each of which endowed with a homogeneous weighted subcoercive system $L_{l,1},\dots,L_{l,n_l}$, for $l=1,\dots,\ell$. Let $(\epsilon_{l,t})_{t>0}$ be the dilations on $\R^{n_l}$ associated with the system $L_{l,1},\dots,L_{l,n_l}$, and define the multi-parameter dilations
\[\gimel_{\vec{t}} = \epsilon_{1,t_1} \times \dots \times \epsilon_{\ell,t_\ell}\]
on $\R^{\vec{n}} = \R^{n_1} \times \dots \times \R^{n_\ell}$; set moreover $X_{\vec{n}} = \{ \lambda \in \R^{\vec{n}} \tc |\lambda_1| \cdots |\lambda_\ell| = 0\}$. If $G$ is a connected Lie group, $\upsilon_l : G_l \to G$ are Lie group homomorphisms, and $L_{l,j}^\flat = \upsilon_l'(L_{l,j})$, then we have a system
\begin{equation}\label{eq:bigsystem}
L_{1,1}^\flat,\dots,L_{1,n_1}^\flat,\dots,L_{\ell,1}^\flat,\dots,L_{\ell,n_\ell}^\flat
\end{equation}
of essentially self-adjoint, left-invariant differential operators on $G$. Under the sole hypothesis of (strong) commutativity of the operators \eqref{eq:bigsystem} on $L^2(G)$, we prove a multi-variate analogue of the previous theorem: the operator $m(L^\flat)$ is bounded on $L^p(G)$ for $1 < p < \infty$ whenever the multiplier $m : \R^{\vec{n}} \to \C$ vanishes on $X_{\vec{n}}$ and satisfies an $L^q$ \emph{Marcinkiewicz condition} of order $\vec s = (s_1,\dots,s_\ell)$, i.e.,
\begin{equation}\label{eq:mkcondition}
\|m\|_{M_\gimel S^{\vec{s}}_q W} \defeq \sup_{t_1,\dots,t_\ell > 0} \| (m \circ \gimel_{\vec t}) \,\eta_1 \otimes \dots \otimes \eta_\ell \|_{S^{\vec{s}}_q W} < \infty,
\end{equation}
with $q \in [2,\infty]$ and
\begin{equation}\label{eq:mkthreshold}
s_l > \frac{Q_{G_l}}{2} + \frac{n_l-1}{q} \qquad\text{for $l=1,\dots,\ell$,}
\end{equation}
where $S^{\vec{s}}_q W(\R^{\vec{n}})$ is the $L^q$ Sobolev space with dominating mixed smoothness of order $\vec{s}$ (see, e.g., \cite{schmeisser_topics_1987,schmeisser_recent_2007}) and the $\eta_l : \R^{n_l} \to \R$ are cut-off functions as before. This result can also be improved in particular cases: in fact, each of the components \eqref{eq:mkthreshold} of the regularity threshold can be independently lowered, by the same techniques and amounts as for the threshold \eqref{eq:mhthreshold} in the previous result. The hypothesis $m|_{X_{\vec{n}}} = 0$, related to the possibility that the spectral measure of $X_{\vec{n}}$ is not null, can be relaxed too, by applying iteratively our Marcinkiewicz-type result to subsystems of \eqref{eq:bigsystem}.

Both our theorems can be applied to the direct-product setting of \cite{sikora_multivariable_2009} (in the case of homogeneous groups), and also to the systems of operators considered in the above-mentioned works on H-type groups; in fact, the results of \cite{mller_marcinkiewicz_1996,veneruso_marcinkiewicz_2000} are sharper than ours, since they require a weaker condition on the multiplier. On the other hand, our theorems have a much wider range of applicability, with respect both to the groups and to the systems of differential operators under consideration. In particular, the environment group $G$ in the second theorem need not be homogeneous. As a corollary, we obtain $L^p$ multiplier theorems for distinguished sublaplacians on some non-nilpotent solvable Lie groups $G$ of polynomial growth (such as the plane motion group, the oscillator groups, the diamond groups) with regularity threshold $(\dim G)/2$; to our knowledge, this threshold had been previously reached only for some homogeneous groups (i.e., Heisenberg and related groups \cite{hebisch_multiplier_1993,mller_spectral_1994}) and for the compact group $SU_2$ \cite{cowling_spectral_2001}.

\subsection*{Notation and preliminary remarks}
For a topological space $X$, we denote by $C(X)$ the space of continuous (complex-valued) functions on $X$, whereas $C_0(X)$ is the subspace of continuous functions vanishing at infinity. If $X$ is a smooth manifold, then $\D(X)$ is the space of compactly supported smooth functions on $X$.

If $G$ is a Lie group, $f$ is a function on $G$ and $x,y \in G$, then we set
\[\LA_x f(y) = f(x^{-1}y) , \qquad \RA_x f(y) = f(y x).\]
$\RA : x \mapsto \RA_x$ is the (right) regular representation of $G$. For a fixed right Haar measure $\mu$ on $G$, $\RA_x$ is an isometry of $L^p(G)$ for $1 \leq p \leq \infty$. With respect to such measure, convolution and involution take the form
\[f * g(x) = \int_G f(xy^{-1}) g(y) \,dy, \qquad f^*(x) = \Delta(x) \overline{f(x^{-1})}\]
(where $\Delta$ is the modular function) and we set, for every representation $\pi$,
\[\pi(f) = \int_G f(x) \, \pi(x^{-1}) \,dx\]
(differently from the common usage), so that in particular
\[\RA(g)f = f * g, \qquad \pi(f * g) = \pi(g) \pi(f), \qquad \pi(Df) = d\pi(D) \pi(f)\]
for every $D \in \Diff(G)$. We denote by $D^+$ the formal adjoint of a smooth differential operator $D$ on $G$ (with respect to the measure $\mu$).

The above conditions \eqref{eq:mhcondition}, \eqref{eq:mkcondition} on the multiplier $m$ have been expressed in terms of Sobolev norms. In fact, there are several scales of spaces which can be used to express a differentiability condition of fractional order and with an $L^q$ flavour; since the inequalities \eqref{eq:mhthreshold}, \eqref{eq:mkthreshold} involving the order are strict, it does not really matter which of the various scales is used. In the following, we will use the scale of Besov spaces $B_{q,r}^s(\R^n)$ and their dominating-mixed-smoothness variants $S^{\vec{s}}_{q,r}B(\R^{\vec{n}})$, which are particularly convenient because of their embedding and interpolation properties (see, e.g., \cite{bergh_interpolation_1976,triebel_interpolation_1978,triebel_theory_1983,schmeisser_topics_1987,runst_sobolev_1996,schmeisser_spaces_2004,schmeisser_recent_2007,sickel_tensor_2009} for a reference); accordingly, we will use the quantities $\|\cdot\|_{M_\epsilon B_{q,q}^s}$, $\|\cdot\|_{M_\gimel S_{q,q}^{\vec{s}}B}$ in place of the ones in \eqref{eq:mhcondition}, \eqref{eq:mkcondition}.

\section{Weighted estimates}\label{section:weightedestimates}

Let $G$ be a Lie group of polynomial growth of degree $Q_G$, and set
\[|x|_G = d_G(x,e), \qquad \langle x \rangle_G = 1 + |x|_G,\]
where $d_G$ is a left-invariant connected distance on $G$ \cite[\S III.4]{varopoulos_analysis_1992} and $e \in G$ is the identity element. Let $L_1,\dots,L_n$ be a weighted subcoercive system on $G$, with associated Plancherel measure $\sigma$, and let $\mathcal{O}$ be the subalgebra of $\Diff(G)$ generated by them. The aim of this section is to obtain inequalities of the form
\[\| \langle \cdot \rangle_G^\alpha \, \Kern_L m \|_{L^p(G)} \leq C_{K,\alpha,\beta} \|m\|_{B_{q,q}^\beta(\R^n)}\]
for multipliers $m$ with support contained in a fixed compact $K \subseteq \R^n$, and for suitable $p,q,\alpha,\beta$. In order to do this, we will subsequently perform two ``changes of variables'' on the spectral side: the former corresponds to choosing a system of generators of $\mathcal{O}$ made of positive weighted subcoercive operators, while the latter is the multi-variate analogue of an exponential change of variables which has often been used in the literature (see, e.g., \cite[\S 6.B]{folland_hardy_1982}) and which allows, together with a Fourier-series decomposition, heat kernel estimates, and H\"older's inequality, to obtain the required estimates. Properties of the Plancherel measure $\sigma$ and interpolation will then be used (as in \cite{mauceri_vectorvalued_1990} and subsequent works) to improve the obtained inequalities.

As in \cite[\S3.2]{martini_spectral}, we can find a polynomial $p_* \geq 0$ on $\R^n$ such that, if
\[p_0(\lambda) = p_*(\lambda) + \sum_j \lambda_j^2+ 1, \qquad p_\nu(\lambda) = p_0(\lambda) + \lambda_\nu \quad\text{for $\nu=1,\dots,n$,}\]
then $p_*(L),p_0(L),p_1(L),\dots,p_n(L)$ are all positive and weighted subcoercive, and moreover $p_0(L),\dots,p_n(L)$ generate $\mathcal{O}$. Let $p : \R^n \to \R^{1+n}$ be the map whose components are the polynomials $p_0,\dots,p_n$. For $l \in \Z^{1+n}$, set
\[E_l(\lambda) = e^{i l \cdot e^{-p(\lambda)}} - 1 = e^{i (l_0 e^{-p_0(\lambda)} + l_1 e^{-p_1(\lambda)} + \dots + l_n e^{-p_n(\lambda)})} - 1.\]
Then $E_l \in C_0(\R^n)$, and in fact
\[E_l = \sum_{0 \neq k \in \N^{1+n}} \frac{(il_0)^{k_0} \cdots (il_n)^{k_n}}{k_0! \cdots k_n!} e^{-k_0 p_0} \cdots e^{-k_n p_n},\]
with uniform convergence on $\R^n$. This means that, if $h_{\nu,t}$ is the heat kernel of $p_\nu(L)$ for $\nu =0,\dots,n$ (with $h_{\nu,0}$ denoting the Dirac delta at the identity of $G$), then
\begin{equation}\label{eq:convolutionseries}
\breve E_l = \sum_{0 \neq k \in \N^{1+n}} \frac{(il_0)^{k_0} \cdots (il_n)^{k_n}}{k_0! \cdots k_n!} h_{0,k_0} * \cdots * h_{n,k_n}
\end{equation}
with convergence in the norm of convolution operators of $L^2(G)$.

\begin{lem}\label{lem:El2}
There exists $C > 0$ such that
\[\|\breve E_l\|_2 \leq C |l| \qquad\text{for all $l \in \Z^{1+n}$.}\]
\end{lem}
\begin{proof}
We have
\[|E_l(\lambda)| \leq |l \cdot e^{-p(\lambda)}| \leq \sum_{\nu = 0}^n |l_j| e^{-p_j(\lambda)} \leq (1+n) |l| e^{-p_*(\lambda)},\]
so that in particular, if $f = e^{p_*} E_l$, then
\[\|\breve E_l\|_2 = \|f(L) \Kern_L (e^{-p_*}) \|_2 \leq \|f\|_\infty \|\Kern_L(e^{-p_*})\|_2 \leq (1+n) \|\Kern_L(e^{-p_*})\|_2 |l|,\]
which is the conclusion.
\end{proof}

\begin{lem}\label{lem:El2exp}
There exist $c,\omega > 0$ such that
\[\|\breve E_l \|_{L^2(G, e^{2|x|_G}\,dx)} \leq c e^{\omega |l|} \qquad\text{for all $l \in \Z^{1+n}$.}\]
\end{lem}
\begin{proof}
Since all the connected left-invariant distances on $G$ are equivalent in the large \cite[Proposition III.4.2]{varopoulos_analysis_1992}, by interpolating the inequalities (e) and (f) of \cite[Theorem~2.3]{martini_spectral}, we have that there exist $c \geq 1$ and $\omega > 0$ such that
\[\|h_{\nu,t} \,e^{|\cdot|_G}\|_{q} \leq c e^{\omega t} \qquad\text{for $t \geq 1$, $\nu = 0,\dots,n$ and $q \in [1,\infty]$.}\]
By Young's inequality and submultiplicativity of $e^{|\cdot|_G}$, we then get
\[\|(h_{0,k_0} * \cdots * h_{n,k_n}) \,e^{|\cdot|_G}\|_{q} \leq c^{1+n} e^{\omega(k_0 + \dots + k_n)}\]
for $k \in \N^{1+n} \setminus \{0\}$ and $q \in [1,\infty]$. This means in particular that the series in \eqref{eq:convolutionseries} converges absolutely in $L^2(G,e^{2|x|_G}\,dx)$, with
\[\sum_{0 \neq k \in \N^{1+n}} \left\| \frac{(il_0)^{k_0} \cdots (il_n)^{k_n}}{k_0! \cdots k_n!} h_{0,k_0} * \cdots * h_{n,k_n} \right\|_{L^2(G,e^{2|x|_G}\,dx)} \leq c^{1+n} e^{e^{\omega} |l|},\]
and we are done.
\end{proof}

\begin{lem}\label{lem:El2poly}
For all $\alpha \geq 0$, we have
\[\|\breve E_l\|_{L^2(G,\langle x \rangle_G^{2\alpha} \,dx)} \leq C_\alpha |l|^{\alpha+1} \qquad\text{for $l \in \Z^{1+n}$.}\]
\end{lem}
\begin{proof}
By Lemma~\ref{lem:El2exp}, it is sufficient to check the estimate for $|l|$ large, but then
\begin{multline*}
\int_{G} |\breve E_l(x)|^2 \langle x \rangle_G^{2\alpha} \,dx \leq \int_{|x|_G \leq \omega|l|} + \int_{|x|_G > \omega|l|} \\
\leq (1+\omega|l|)^{2\alpha} \|\breve E_l\|_2^2 + \sup_{r > \omega|l|} {\textstyle \frac{(1+r)^{2\alpha}}{e^{2r}}} \|\breve E_l\|_{L^2(G,e^{2|x|_G}\,dx)}^2 
\leq C_\alpha |l|^{2(\alpha+1)}
\end{multline*}
by Lemmata~\ref{lem:El2} and \ref{lem:El2exp}.
\end{proof}

\begin{lem}\label{lem:fourierdecomposition}
Let $K \subseteq \R^n$ be compact. For every $f \in \D(\R^n)$ supported in $K$, there exists $g \in \D(\T^{1+n})$, depending linearly on $f$, such that
\begin{gather}
\label{eq:fdec1} f(\lambda) = g\left(e^{i e^{-p(\lambda)}}\right) = g\left(e^{i e^{-p_0(\lambda)}}, \dots, e^{i e^{-p_n(\lambda)}}\right),\\
\label{eq:fdec2} g(1,\dots,1) = 0,\\
\label{eq:fdec3} \|g\|_{B_{2,2}^s(\T^{1+n})} \leq C_{K,s} \|f\|_{B_{2,2}^s(\R^n)} \qquad\text{for all $s \geq 0$.}
\end{gather}
In particular, if $g(e^{it}) = \sum_{l \in \Z^{1+n}} \hat{g}(l) e^{i l \cdot t}$ is the Fourier-series development of $g$, then we have
\begin{equation}\label{eq:fdec4}
f = \sum_{0 \neq l \in \Z^{1+n}} \hat g(l) E_l,
\end{equation}
with uniform convergence on $\R^n$.
\end{lem}
\begin{proof}
Since $K \subseteq \R^n$ is compact and the polynomials $p_0,\dots,p_n$ are strictly positive, $p(K)$ is a compact subset of $\Omega = \left]0,+\infty\right[^{1+n}$. Therefore we can choose $\psi_K \in \D(\Omega)$ such that $\psi_K|_{p(K)} \equiv 1$. If we put
\[\tilde f(y) = f(y_1 - y_0,\dots,y_n - y_0) \psi_K(y) \qquad\text{for $y \in \R^{1+n}$,}\]
we then have that $\tilde f \in \D(\Omega)$, $f = \tilde f \circ p$ and
\[\|\tilde f\|_{B_{2,2}^s(\R^{1+n})} \leq C_{K,s} \|f\|_{B_{2,2}^s(\R^n)} \qquad\text{for all $s \geq 0$,}\]
since the change of variables has maximal rank.

Notice now that the map
\[\Phi : \Omega \ni y \mapsto e^{i e^{-y}} = (e^{i e^{-y_0}},\dots,e^{i e^{-y_n}}) \in \T^{1+n}\]
is a smooth diffeomorphism with its image, which is an open subset of $\T^{1+n}$ not containing $(1,\dots,1)$. The function $g = \tilde f \circ \Phi^{-1} \in \D(\Phi(\Omega))$ can be then extended by zero to a smooth function on $\T^{1+n}$, and we have clearly
\[\|g\|_{B_{2,2}^s(\T^{1+n})} \leq C_{K,s} \|\tilde f\|_{B_{2,2}^s(\R^{1+n})} \qquad\text{for all $s \geq 0$.}\]

The construction shows that $g$ depends linearly on $f$ and satisfies \eqref{eq:fdec1}-\eqref{eq:fdec3}. In particular, we have $\sum_{l \in \Z^{1+n}} \hat g(l) = 0$, so that the Fourier decomposition of $g$ can be rewritten as
\[g(e^{it}) = \sum_{0 \neq l \in \Z^{1+n}} \hat{g}(l) (e^{i l \cdot t} - 1)\]
(with uniform convergence since $g$ is smooth), which gives \eqref{eq:fdec4}.
\end{proof}

\begin{prp}\label{prp:sobolevl2estimates}
Let $K \subseteq \R^n$ be compact, $\alpha \geq 0$, $\beta > \alpha + (n+3)/2$. For all $f \in \D(\R^n)$ with $\supp f \subseteq K$, we have
\[\|\breve f\|_{L^2(G, \langle x \rangle_G^{2\alpha} \,dx)} \leq C_{K,\alpha,\beta} \|f\|_{B_{2,2}^\beta(\R^n)}.\]
\end{prp}
\begin{proof}
Let $g \in \D(\T^{1+n})$ be given by Lemma~\ref{lem:fourierdecomposition}. Then 
\[\breve f = \sum_{0 \neq l \in \Z^{1+n}} \hat g(l) \breve E_l\]
in the norm of convolution operators on $L^2(G)$. However, the series in the right-hand side converges absolutely in $L^2(G,\langle x \rangle_G^{2\alpha} \,dx)$, since
\begin{multline*}
\sum_{0 \neq l \in \Z^{1+n}} |\hat g(l)| \, \|\breve E_l\|_{L^2(G,\langle x \rangle_G^{2\alpha}\,dx)} \leq C_\alpha \sum_{0 \neq l \in \Z^{1+n}} |\hat g(l)| \, |l|^{\alpha + 1} \\
\leq C_{\alpha,\beta} \|g\|_{B_{2,2}^\beta(\T^{1+n})} \leq C_{K,\alpha,\beta} \|f\|_{B_{2,2}^\beta(\R^n)}
\end{multline*}
by Lemma~\ref{lem:El2poly} and H\"older's inequality, and the conclusion follows.
\end{proof}

The previous proposition contains a ``preliminary version'' of the required inequalities, which we are now going to sharpen by interpolation with the Plancherel formula. In order to control the $L^2(\sigma)$ norm with a Besov norm, we will use a refined trace theorem due to Triebel.

Let $\tau$ be a (positive) regular Borel measure on $\R^n$, and let $0 \leq d \leq n$; we say that $\tau$ is \emph{locally $d$-bounded} on an open $\Omega \subseteq \R^n$ if, for every compact $K \subseteq \Omega$ and for $0 \leq \gamma < d$, there exist $C,\bar r > 0$ such that
\[\tau(B(\lambda,r)) \leq C r^\gamma \qquad\text{for $\lambda \in K$ and $r \leq \bar r$.}\]
Notice that every regular Borel measure $\tau$ is locally $0$-bounded on the whole $\R^n$; moreover, if $\tau$ is homogeneous with respect to some system of dilations $\epsilon_t$ on $\R^n$ (i.e., if $\tau(\epsilon_t(A)) = t^a \tau(A)$ for some $a \geq 0$ and every Borel $A \subseteq \R^n$), then $\tau$ is locally $1$-bounded on $\R^n \setminus \{0\}$.

\begin{lem}\label{lem:triebeltrace}
Let $\tau$ be a regular Borel measure which is locally $d$-bounded on an open $\Omega \subseteq \R^n$. If $s > (n-d)/2$ and $K \subseteq \Omega$ is compact, then
\[\|f\|_{L^2(\tau)} \leq C_{K,s} \|f\|_{B_{2,2}^s(\R^n)}\]
for every $f \in \D(\R^n)$ with $\supp f \subseteq K$.
\end{lem}
\begin{proof}
If $d=0$, then $\|f\|_\infty \leq C_s \|f\|_{B_{2,2}^s(\R^n)}$ and the result is trivial. Suppose instead that $0 < d \leq n$, and let $K \subseteq \Omega$ be compact, $\varepsilon > 0$. Choose a compact neighborhood $K' \subseteq \Omega$ of $K$, and let $C,\bar r > 0$ such that
\[\tau(B(\lambda,r)) \leq C r^{d-\epsilon} \qquad\text{for $\lambda \in K'$ and $0 < r \leq \bar r$.}\]
Let moreover $\bar r' = \min\{\bar r, \mathop{\mathrm{dist}}(\Omega \setminus \mathring{K}', K)\}$, $C' = \max\{C, \tau(K)/(\bar r')^{d-\epsilon}\}$. The identity $\tau_K(E) = \tau(E \cap K)$ defines a positive regular Borel measure $\tau_K$ on $\R^n$, which coincides with $\tau$ on $K$, and with $\supp \tau_K \subseteq K$. Moreover
\[\tau_K(B(\lambda,r)) \leq C' r^{d-\varepsilon} \qquad\text{for every $r > 0$ and $\lambda \in \R^n$,}\]
by construction. Therefore, since the Besov space $B_{2,2}^s(\R^n)$ coincides with the Triebel-Lizorkin space $F_{2,2}^s(\R^n)$, by \cite[Corollary 9.8(ii)]{triebel_structure_2001} we have
\[\|f\|_{L^2(\tau)} = \|f\|_{L^2(\tau_K)} \leq C_{K,s} \|f\|_{B_{2,2}^s(\R^n)}\]
for $s > (n-d)/2 + \varepsilon/2$ and $f \in \D(\R^n)$ with $\supp f \subseteq K$. The conclusion follows from the arbitrariness of $\varepsilon>0$.
\end{proof}

\begin{theo}\label{thm:weightedestimates}
Suppose that, for some open $\Omega \subseteq \R^n$, the Plancherel measure $\sigma$ is locally $d$-bounded on $\Omega$. Let $K \subseteq \Omega$ be compact, $D \in \Diff(G)$, $p,q \in [1,\infty]$, $\alpha \geq 0$,
\[\beta > \alpha + Q_G \left(\frac{1}{\min\{2,p\}} - \frac{1}{2}\right) + \frac{n}{q} - \frac{d}{\max\{2,q\}}.\]
For all $m \in B^{\beta}_{q,q}(\R^n)$ with $\supp m \subseteq K$, we have
\[\|\langle \cdot \rangle_G^\alpha \, D \breve m\|_{L^p(G)} \leq C_{K,D,\alpha,\beta,p,q} \|m\|_{B_{q,q}^\beta(\R^n)}.\]
\end{theo}
\begin{proof}
Consider first the case $p=2$, $D = 1$. Let $\xi \in \D(\R^n)$ be such that $\supp \xi \subseteq \Omega$, $\xi|_K \equiv 1$, and let $K' \subseteq \Omega$ be a compact neighborhood of $\supp \xi$. Proposition~\ref{prp:sobolevl2estimates}, together with the continuous inclusion $B_{q,2}^\beta(\R^n) \subseteq B_{2,2}^{\beta - n/2}(\R^n)$, then yields, for $m \in \D(\R^n)$ with $\supp m \subseteq K'$, that
\begin{equation}\label{eq:prelimineq}
\|\breve m\|_{L^2(G,\langle x \rangle_G^{2\alpha} \,dx)} \leq C_{K,\alpha,\beta,q} \|m\|_{B^\beta_{q,2}(\R^n)}
\end{equation}
for $\beta > \alpha + (2n+3)/2$. By the use of a suitable approximate identity, \eqref{eq:prelimineq} can be easily extended to all $m \in B^\beta_{q,2}(\R^n)$ with $\supp m \subseteq \supp \xi$. Hence, if we consider the linear map $M : m \mapsto \Kern_L (m\xi)$, then we have that
\begin{equation}\label{eq:limitatezzaprelim}
\text{$M$ is bounded } B^\beta_{q,2}(\R^n) \to L^2(G,\langle x \rangle_G^{2\alpha} \,dx) \text{ for $\beta > \alpha + (2n+3)/2$}.
\end{equation}
On the other hand, for $\alpha = 0$, the Plancherel formula and Lemma~\ref{lem:triebeltrace} give
\begin{equation}\label{eq:limitatezzabase}
\text{$M$ is bounded } B^\beta_{q,2}(\R^n) \to L^2(G) \text{ for $\beta > \frac{n}{q} - \frac{d}{\max\{2,q\}}$}
\end{equation}
(this is clear for $q = 2$ and $q= \infty$; for $1 \leq q < 2$, we exploit the continuous inclusion $B^\beta_{q,2}(\R^n) \subseteq B_{2,2}^{\beta - n/q + n/2}(\R^n)$; for $2 < q < \infty$, we interpolate). Therefore, by interpolating \eqref{eq:limitatezzaprelim} and \eqref{eq:limitatezzabase}, we get
\[\text{$M$ is bounded } B^\beta_{q,2}(\R^n) \to L^2(G,\langle x \rangle_G^{2\alpha} \,dx) \text{ for $\beta > \alpha + \frac{n}{q} - \frac{d}{\max\{2,q\}}$}.\]
In order to conclude, it is sufficient to notice that, if $\beta > \alpha+n/q-d/\max\{2,q\}$, then for any $\beta' \in \left]\alpha+n/q-d/\max\{2,q\},\beta\right[$ we have the continuous inclusion $B_{q,q}^\beta(\R^n) \subseteq B_{q,2}^{\beta'}(\R^n)$ , and moreover $M m = \breve m$ for every $m$ with $\supp m \subseteq K$.

Take now an arbitrary $D \in \Diff(G)$. For $m \in B^\beta_{\infty,\infty}(\R^n)$ with $\supp m \subseteq K$, set $m_0 = m e^{p_*}$, $\xi = e^{-p_*}$; then $\breve m = \breve m_0 * \breve\xi$, so that, by Young's inequality,
\[\|\langle \cdot \rangle_G^\alpha \, D \breve m\|_2 \leq \|\langle \cdot \rangle_G^\alpha \, \breve m_0\|_2 \|\langle \cdot \rangle_G^\alpha \, D\breve\xi\|_1 \lesssim \|m_0\|_{B^\beta_{q,q}(\R^n)} \lesssim \|m\|_{B^\beta_{q,q}(\R^n)}\]
for $\beta > \alpha+n/q-d/\max\{2,q\}$. This concludes the proof for the case $p = 2$.

The case $1 \leq p < 2$ follows by H\"older's inequality, since $\int_G \langle x \rangle_G^{-\alpha} \,dx < \infty$ for $\alpha > Q_G$. 

Let now $p=\infty$. If $\zeta \in \D(\R^n)$ is nonnegative and $\zeta(e) > 0$, and if we set $w_\alpha = \langle \cdot \rangle_G^\alpha * \zeta$ for $\alpha \geq 0$, then $w_\alpha$ is smooth and nonnegative,
\[\langle x \rangle_G^\alpha \leq C_\alpha w_\alpha(x), \qquad D w_\alpha(x) \leq C_{D,\alpha} \langle x \rangle_G^\alpha\]
for all $D \in \Diff(G)$. If $\Diff_*$ is a basis of the $A \in \Diff(G)$ of order up to $\lfloor (\dim G)/2 \rfloor + 1$, then we have, for $\beta > \alpha+n/q-d/\max\{2,q\}$ and $D \in \Diff(G)$,
\begin{multline*}
\| \langle \cdot \rangle_G^\alpha \, D \breve m\|_\infty \lesssim \| w_\alpha \, D \breve m\|_\infty \lesssim \sum_{A \in \Diff_*} \| A (w_\alpha D \breve f) \|_2 \\
\lesssim \sum_{A_1,A_2 \in \Diff_*} \| (A_1 w_\alpha) (A_2 D \breve m) \|_2 \lesssim \sum_{A \in \Diff_*} \| \langle \cdot \rangle_G^\alpha \, A D \breve m \|_2 \lesssim \|m\|_{B_{q,q}^\beta}
\end{multline*}
by Sobolev's embedding, Leibniz's rule and the case $p=2$.

The remaining case $2 < p < \infty$ follows by interpolation.
\end{proof}

\section{Improved weighted estimates}\label{section:metivier}

The weighted estimates given by Theorem~\ref{thm:weightedestimates} for $p=1$ yield a ``weak multiplier theorem'' for a weighted subcoercive system $L_1,\dots,L_n$ on a Lie group $G$ of polynomial growth: the operator $m(L)$ is bounded on $L^p(G)$ for $1 \leq p \leq \infty$ if the multiplier $m$ is compactly supported and sufficiently smooth; more precisely, by taking $q = \infty$, we require an order of smoothness $\gamma > Q_G/2$, where $Q_G$ is the dimension at infinity of $G$. If $G = \R^n$, then $Q_G$ coincides with the topological dimension $\dim G = n$; for non-abelian (simply connected) nilpotent groups, however, $Q_G > \dim G$. Nevertheless, for a particular class of $2$-step nilpotent groups (namely, Heisenberg and related groups) multiplier theorems have been proved with $(\dim G)/2$ as the regularity threshold \cite{hebisch_multiplier_1993,mller_spectral_1994,hebisch_multiplier_1995,mller_marcinkiewicz_1996,veneruso_marcinkiewicz_2000}. In this section, we extend to our context of weighted subcoercive systems the technique of Hebisch and Zienkiewicz \cite{hebisch_multiplier_1995}, which allows in some cases to lower the threshold in the weighted $L^1$ estimates.

Let $G$ be a nilpotent Lie group, with Lie algebra $\lie{g}$. Let $\lie{z}$ be the center of $\lie{g}$ and set
\begin{equation}\label{eq:centercenter}
\lie{y} = \{ v \in \lie{g} \tc [v,\lie{g}] \subseteq \lie{z}\};
\end{equation}
$\lie{y}$ is a characteristic ideal of $\lie{g}$ containing $\lie{z}$ (in fact, it is the term following $\lie{z}$ in the ascending central series of $\lie{g}$). Let moreover $P : \lie{g} \to \lie{g}/\lie{z}$ be the canonical projection. The bilinear map $[\cdot,\cdot] : \lie{g} \times \lie{g} \to \lie{g}$ induces, by restriction, passage to the quotient and transposition, another bilinear map
\[J : \lie{g}/\lie{z} \times \lie{z}^* \to \lie{y}^*,\]
which we will call the \emph{capacity map}\index{capacity map} of $\lie{g}$, and is uniquely determined by
\[J(P(x),\tau)(y) = \tau([x,y])\]
for $x \in \lie{g}$, $y \in \lie{y}$, $\tau \in \lie{z}^*$. The group $G$ is said to be an \emph{H-type group} if there exists an inner product on $\lie{g}$ such that, for every $\tau \in \lie{z}^*$ of norm $1$, the map $J(\cdot,\tau) : \lie{g}/\lie{z} \to \lie{y}^*$ is an isometric embedding (this condition implies that $\lie{g} = \lie{y}$, so that $G$ is $2$-step). If $G$ is a H-type group, then in particular
\begin{equation}\label{eq:metivier}
|J(\bar x,\tau)| \geq |\bar x| |\tau|
\end{equation}
for suitable norms on $\lie{g}/\lie{z}$, $\lie{z}^*$ and $\lie{y}^*$; the validity of such an inequality defines the class of \emph{M\'etivier groups}\index{Lie group!M\'etivier}, which has been introduced in the study of analytic hypoellipticity of Rockland operators \cite{metivier_hypoellipticite_1980,helffer_conditions_1982}; this class is strictly larger than that of H-type groups (see \cite{mller_singular_2004} for an example), but is still contained in the class of $2$-step groups.

In the following, we consider a more general inequality of the form
\[|J(\bar x,\tau)| \geq w(\bar x) \zeta(\tau)\]
for some non-negative functions $w : \lie{g}/\lie{z} \to \R$, $\zeta : \lie{z}^* \to \R$, which may hold also on higher-step groups. Rewritten as
\[w(\bar x)^{\gamma} \leq |J(\bar x,\tau)|^\gamma \zeta(\tau)^{-\gamma}\]
for some $\gamma > 0$, this inequality will be interpreted via the spectral theorem, in order to control a multiplication operator (corresponding to $w(\bar x)^\gamma$) with a function of the central derivatives (corresponding to $\zeta(\tau)^{-\gamma})$; in this interpretation, it turns out that $|J(\bar x,\tau)|^2$ corresponds to a sum of products of left- and right-invariant differential operators on $G$, therefore the term $|J(\bar x,\tau)|^\gamma$ can be dominated by an a priori estimate for a weighted subcoercive operator on the direct product $G \times G$.

In order to fill in the details, it is convenient to introduce some notation. For every smooth differential operator $D$ on $G$, the identity
\begin{equation}\label{eq:dstar}
(Df)^* = D^\dstar f^*
\end{equation}
defines another differential operator $D^\dstar$ on $G$; the map $D \mapsto D^\dstar$ is a conjugate-linear involutive automorphism of the unital algebra of all smooth differential operators on $G$, which maps left-invariant operators to right-invariant ones and vice versa.

The Lie algebra $\lie{\tilde g}$ of the direct product $\tilde G = G \times G$ is canonically isomorphic to $\lie{g} \oplus \lie{g}$; we define the correspondence $D \mapsto D^\dtwist$ on $\Diff(\tilde G)$ as the unique conjugate-linear automorphism of the unital algebra $\Diff(\tilde G) \cong \UEnA(\lie{g} \oplus \lie{g})$ extending the Lie algebra automorphism $(X,Y) \mapsto (Y,X)$ of $\lie{g} \oplus \lie{g}$.

Let $\xi$ be the unitary representation of $\tilde G$ on $L^2(G)$ given by $\xi(x,y) f = \RA_x \LA_y f$. Then, for every $D \in \Diff(\tilde G)$, $d\xi(D)$ is a smooth differential operator on $G$, and
\[d\xi(D^\dtwist) = d\xi(D)^\dstar.\]
Finally, for $D \in \Diff(G)$, let $\tilde D \in \Diff(\tilde G)$ be defined by $\tilde D(f \otimes g) = (Df) \otimes g$, so that in particular $d\xi(\tilde D) = D$.

\begin{lem}\label{lem:twistrockland}
Let $L = L^+ \in \Diff(G)$ be weighted subcoercive, and set $\Delta = L^2$. Then $\tilde \Delta + \tilde \Delta^\dtwist$ is positive weighted subcoercive on $\tilde G$.
\end{lem}
\begin{proof}
For $D \in \Diff(G)$, let $\overline{D} \in \Diff(G)$ be the differential operator uniquely determined by the identity $\overline{D f} = \overline{D} \, \overline{f}$. The map $D \mapsto \overline{D}$ defines a conjugate-linear involutive automorphism of the unital algebra $\Diff(G)$, and it is easily proved that $\tilde{D}^\dtwist(f \otimes g) = f \otimes (\overline{D} g)$. In particular, we have
\[(\tilde \Delta + \tilde \Delta^\dtwist)(f \otimes g) = (L^2 f) \otimes g + f \otimes (\overline{L}^2 g) \]
In view of \cite[Theorem~5.4]{martini_spectral}, since $L$ is self-adjoint and weighted subcoercive, in order to conclude it will be sufficient to show that $\overline{L}$ is weighted subcoercive too.

As in \cite[\S2]{martini_spectral}, fix a weighted structure of $\lie{g}$ and a weighted subcoercive form $C$ such that $d\RA_G(C) = L$. If $\overline{C}$ is the form defined by $\overline{C}(\alpha) = \overline{C(\alpha)}$, then it is easy to see that $\overline{L} = d\RA_G(\overline{C})$, and, on the other hand,
\[\Re \langle \phi, d\RA_G(\overline{C}) \phi \rangle = \Re \langle \overline{\phi}, d\RA_G(C) \overline{\phi} \rangle,\]
thus $\overline{C}$ is also weighted subcoercive by definition.
\end{proof}

Let $L_1,\dots,L_n \in \Diff(G)$ be a weighted subcoercive system on the nilpotent Lie group $G$, and let $\Delta = p(L)^2$, where $p$ is a real polynomial such that $p(L)$ is weighted subcoercive. We define
\[\tilde A = (\tilde \Delta + \tilde \Delta^\dtwist)/2, \qquad A = d\xi(\tilde A) = (\Delta + \Delta^\dstar)/2.\]
By Lemma~\ref{lem:twistrockland}, $\tilde A$ is a (left-invariant) positive weighted subcoercive operator on $\tilde G$, whereas $A$ is a differential operator on $G$ which in general is neither left- nor right-invariant; since $\tilde A, \tilde \Delta, \tilde \Delta^\dtwist$ form a weighted subcoercive system, the corresponding operators $A,\Delta,\Delta^\dstar$ in the representation $\xi$ admit a joint spectral resolution.

Let $h_t$ ($t > 0$) be the convolution kernel of $e^{-t\Delta}$.

\begin{lem}\label{lem:kerneloperations}
Suppose that $u \in L^2(G)$ commutes with all the $h_t$ ($t > 0$). For all Borel $m : \R \to \C$, $u$ is in the domain of $m(\Delta)$ if and only if it is in the domain of $m(A)$, and in this case
\begin{equation}\label{eq:kerneloperations}
m(A) u = m(\Delta) u.
\end{equation}
\end{lem}
\begin{proof}
From \eqref{eq:dstar} we easily deduce $e^{-t\Delta^\dstar} f = (f^* * h_t)^* = h_t * f$ and
\[e^{-tA} f = e^{-t\Delta/2} e^{-t\Delta^\dstar/2} f = h_{t/2} * f * h_{t/2},\]
so that $e^{-tA} u = e^{-t\Delta} u$. If $\xi_t(\lambda) = e^{-t \lambda}$ and $\JJ_0 = \Span \{\xi_t \tc t > 0\}$, then we obtain \eqref{eq:kerneloperations} for $m \in \JJ_0$. It is not difficult to extend \eqref{eq:kerneloperations} to $m \in C_0(\R^n)$ by the Stone-Weierstrass theorem, and then to all Borel $m : \R \to \C$ by the spectral theorem and dominated convergence.
\end{proof}

\begin{lem}\label{lem:leftrightdiff}
Let $X \in \lie{g}$. Then, for all $v \in \lie{g}$,
\[(X + X^\dstar)|_{\exp(v)} = d\exp_v([v,X]).\]
\end{lem}
\begin{proof}
The semigroup associated to $\tilde X + \tilde X^\dtwist$ is $t \mapsto (\exp(tX),\exp(tX))$, so that, for all $f \in \D(G)$, $v \in \lie{g}$,
\[(X+X^\dstar)|_{\exp(v)} f = \left.\frac{d}{d t}\right|_{t=0} f(\exp(-tX) \exp(v) \exp(tX)).\]
Since $\exp(-tX) \exp(v) \exp(tX) = \exp(\Ad(\exp(-tX))(v))$, we have
\[\left.\frac{d}{d t}\right|_{t=0} (\exp(-tX) \exp(v) \exp(tX)) = d\exp_v(\ad(-X)(v)) = d\exp_v([v,X]),\]
which is the conclusion.
\end{proof}

In the following, we will identify $G$ with $\lie{g}$ via the exponential map. Choose a basis $\nu_1,\dots,\nu_r$ of $(\lie{g}/\lie{z})^*$ and a basis $T_1,\dots,T_d$ of $\lie{z}$, and set $P_j = \nu_j \circ P$. The functions $P_j : G \to \R$ can be thought of as multiplication operators on $L^2(G)$, and it is not difficult to show that the operators
\[P_1,\dots,P_r,-iT_1,\dots,-iT_d\]
are (essentially) self-adjoint on $L^2(G)$ and commute strongly pairwise, so that they admit a joint spectral resolution.

Through the chosen bases, $J$ can be identified with a bilinear map $\R^r \times \R^d \to \lie{y}^*$. Therefore, for every $Y \in \lie{y}$, we have a bilinear form $J(\cdot,\cdot)(Y) : \R^r \times \R^d \to \R$, which in fact is a polynomial; we can then evaluate this particular polynomial in the operators $P_1,\dots,P_r,-iT_1,\dots,-iT_d$, and denote by $J(P,-iT)(Y)$ the resulting operator on $L^2(G)$. Finally, choose an inner product on $\lie{y}$ (which induces an inner product on $\lie{y}^*$) and an orthonormal basis $\{Y_l\}_l$ of $\lie{y}$; then also the map $|J(\cdot,\cdot)|^2$ is a polynomial, thus as before we can consider the operator $|J(P,-iT)|^2$ on $L^2(G)$, and clearly $|J(P,-iT)|^2 = \sum_l (J(P,-iT)(Y_l))^2$.

\begin{lem}\label{lem:jdiffop}
For all $Y \in \lie{y}$, $J(P,-iT)(Y)$ is a differential operator on $G$; more precisely, $J(P,-iT)(Y) = -i(Y + Y^\dstar)$. In particular
\[|J(P,-iT)|^2 = -\sum_l (Y_l + Y_l^\dstar)^2 = d\xi\left( - \sum_l (\tilde Y_l + \tilde Y_l^\dtwist)^2 \right).\]
\end{lem}
\begin{proof}
Let $\hat T_1,\dots,\hat T_d \in \lie{z}^*$ and $\hat \nu_1,\dots,\hat \nu_r \in \lie{g}/\lie{z}$ be the dual bases of $T_1,\dots,T_d$ and $\nu_1,\dots,\nu_r$ respectively. Then, by bilinearity, for every $Y \in \lie{y}$,
\[J(P,-iT)(Y) = -i \sum_{j,k} J(\hat \nu_j,\hat T_k)(Y) P_j T_k.\]
This shows that $J(P,-iT)(Y)$ is a differential operator on $G$. In fact, for all $x \in G = \lie{g}$, we have $\sum_j P_j(x) \hat\nu_j = P(x)$,
therefore
\[\begin{split}
J(P,-iT)(Y)|_x  &= -i \sum_k J(P(x),\hat T_k)(Y) T_k \\
&= -i \sum_k \hat T_k([x,Y]) T_k = -i [x,Y] = -i (Y + Y^\dstar)|_x
\end{split}\]
by Lemma~\ref{lem:leftrightdiff} (notice that, since $T_1,\dots,T_d$ are central, they are constant vector fields in exponential coordinates).
\end{proof}

Since $T_1,\dots,T_d$ are central, the left-invariant differential operators
\begin{equation}\label{eq:centralsystem}
L_1,\dots,L_n,-iT_1,\dots,-iT_d
\end{equation}
on $G$ are a weighted subcoercive system. We can thus consider the Planche\-rel measure $\sigma'$ on $\R^n \times \lie{z}^*$ associated to this system, which can be shown not to depend on the choice of the basis of $\lie{z}$.

The core of the technique under discussion is contained in the following

\begin{prp}\label{prp:partialweight}
Suppose that, for some nonnegative Borel functions $w : \lie{g}/\lie{z} \to \R$ and $\zeta : \lie{z}^* \to \R$, we have
\[|J(\bar x,\tau)| \geq w(\bar x) \, \zeta(\tau) \qquad \text{for all $\bar x \in \lie{g}/\lie{z}$, $\tau \in \lie{z}^*$.}\]
If $K \subseteq \R^n$ is compact and $\gamma \geq 0$, then, for all $m \in \D(\R^n)$ with $\supp m \subseteq K$,
\[\||w \circ P|^\gamma \breve m\|_2^2 \leq C_{K,\gamma} \int_{\R^n \times \lie{z}^*} |m(\lambda)|^2 \,\zeta(\tau)^{-2\gamma} \,d\sigma'(\lambda,\tau).\]
\end{prp}
\begin{proof}
From the hypothesis we deduce, by the spectral theorem,
\[\| |w \circ P|^\gamma f  \|_2 \leq C_\gamma \| |J(P,-iT)|^\gamma \zeta(-iT)^{-\gamma} f\|_2\]
for $f \in L^2(G)$. By Lemma~\ref{lem:jdiffop}, $|J(P,-iT)|^2 = d\xi(D)$ for some $D \in \Diff(\tilde G)$; since $\tilde A$ is weighted subcoercive on $\tilde G$, by Theorem~2.3(iii) of \cite{martini_spectral}, for some polynomial $q_\gamma$ we have, in the representation $\xi$,
\[\| |J(P,-iT)|^\gamma \psi \|_2 \leq C_\gamma \| q_\gamma(A) \psi\|_2,\]
therefore, by putting the two inequalities together, we get
\[\| |w \circ P|^\gamma f \|_2 \leq C_\gamma \|\zeta(-iT)^{-\gamma} q_\gamma(A) f \|_2\]
(since the $T_j$ commute strongly with $A$). In particular, if we take $f = \breve m$,
\[\| |w \circ P|^\gamma \breve m \|_2 \leq C_\gamma \|\zeta(-iT)^{-\gamma} q_\gamma(A) \breve m \|_2 = C_\gamma \|\zeta(-iT)^{-\gamma} q_\gamma(\Delta) \breve m \|_2\]
by Lemma~\ref{lem:kerneloperations}, since $\breve m$ commutes with all the $h_t$. On the other hand, by the Plancherel formula for the system \eqref{eq:centralsystem},
\[\|\zeta(-iT)^{-\gamma} q_\gamma(\Delta) \breve m \|_2^2 \leq C_{K,\gamma} \int_{\R^n \times \lie{z}^*} |\breve m(\lambda)|^2 \, \zeta(\tau)^{-2\gamma} \,d\sigma'(\lambda,\tau),\]
where $C_{K,\gamma}  = \sup_{\lambda \in K} q_\gamma(p(\lambda)^2)^2$, and we are done.
\end{proof}

Simple manipulations give a slightly more general form of the previous estimate:

\begin{cor}\label{cor:partialweight}
Suppose that, for some nonnegative Borel functions $w_j : \lie{g}/\lie{z} \to \R$ and $\zeta_j : \lie{z}^* \to \R$ ($j=1,\dots,h$), we have
\[|J(\bar x,\tau)| \geq w_j(\bar x) \, \zeta_j(\tau) \qquad \text{for all $\bar x \in \lie{g}/\lie{z}$, $\tau \in \lie{z}^*$,}\]
and set $\tilde w_j(x) = 1 + w_j(P(x))$. If $K \subseteq \R^n$ is compact, then for all $m \in \D(\R^n)$ with $\supp m \subseteq K$ and for all $\vec{\gamma} = (\gamma_1,\dots,\gamma_h) \geq 0$ we have
\[\| \tilde w_1^{\gamma_1} \cdots \tilde w_h^{\gamma_h} \, \breve m \|^2_{L^2(G)} 
\leq C_{K,\vec\gamma} \int_{\R^n \times \lie{z}^*} |m(\lambda)|^2 \,\prod_{j=1}^h (1+\zeta_j(\tau)^{-2\gamma_j}) \,d\sigma'(\lambda,\tau).\]
\end{cor}
\begin{proof}
If we set, for $I \subseteq \{1,\dots,h\}$,
\[\gamma_I = \sum_{j\in I} \gamma_j, \qquad w_{\vec{\gamma},I}(\bar x) = \prod_{j \in I} w_j(\bar x)^{\gamma_j/\gamma_I}, \qquad \zeta_{\vec{\gamma},I}(\tau) = \prod_{j \in I} \zeta_j(\tau)^{\gamma_j/\gamma_I},\]
then clearly
\[|J(\bar x,\tau)| \geq w_I(\bar x) \, \zeta_I(\tau)  \qquad \text{for all $\bar x \in \lie{g}/\lie{z}$, $\tau \in \lie{z}^*$,}\]
and moreover
\[\prod_{j=1}^h \tilde w_j^{2\gamma_j} \leq C_{\vec{\gamma}} \sum_{I \subseteq \{1,\dots,h\}} (w_{\vec{\gamma},I} \circ P)^{2\gamma_I}, \quad \prod_{j=1}^h (1+\zeta_j^{-2\gamma_j}) = \sum_{I \subseteq \{1,\dots,h\}} \zeta_{\vec{\gamma},I}^{-2\gamma_I},\]
therefore the conclusion follows by repeated application of Proposition~\ref{prp:partialweight}.
\end{proof}

Under some particular hypotheses, we may therefore control a weighted $L^2$ norm of $\breve m$ in terms of an $L^2(\sigma_{\tilde\zeta})$ norm of $m$, where $\sigma_{\tilde\zeta}$ is the push-forward of
\begin{equation}\label{eq:measurenearlyproduct}
\tilde\zeta(\tau) \,d\sigma'(\lambda,\tau)
\end{equation}
on the first factor of $\R^n \times \lie{z}^*$, for some nonnegative function $\tilde\zeta : \lie{z}^* \to \R$.

\begin{lem}\label{lem:homogeneouspushforward}
(i) Suppose that $\tilde\zeta \in L^1_\loc(\lie{z}^*)$ is nonnegative. Then $\sigma_{\tilde\zeta}$ is a regular Borel measure on $\R^n$.

(ii) Suppose moreover that $G$ is a homogeneous group, with dilations $\delta_t$ and homogeneous dimension $Q_\delta$, and that $L_1,\dots,L_n$ is a homogeneous system, with associated dilations $\epsilon_t$. If $\tilde\zeta$ is homogeneous of degree $a$, i.e., $\tilde\zeta(\tau \circ \delta_t) = t^a \tilde\zeta(\tau)$, then $\sigma_{\tilde\zeta}$ is homogeneous of degree $Q_\delta+a$, i.e., $\sigma_{\tilde\zeta}(\epsilon_t(A)) = t^{Q_\delta+a} \sigma_{\tilde\zeta}(A)$.
\end{lem}
\begin{proof}
(i) Let $K \subseteq \R^n$ be compact. By \cite[Lemma 3.16]{martini_spectral}, the canonical projection $\R^n \times \lie{z}^* \to \R^n$ is a proper continuous map when restricted to $\supp \sigma'$, hence there is a compact $K' \subseteq \lie{z}^*$ such that $(K \times \lie{z}^*) \cap \supp \sigma' \subseteq K \times K'$, and consequently
\[\sigma_{\tilde\zeta}(K) \leq C_K \int_{K \times K'} e^{-2p(\lambda)^2} \,\tilde\zeta(\tau) \,d\sigma'(\lambda,\tau)
= C_K \| (\tilde\zeta \chr_{K'})^{1/2}(-iT) h_1\|_2^2,\]
by the Plancherel formula. On the other hand, since $h_1$ is in the Schwartz class, the last quantity is easily seen to be finite by using the Euclidean Fourier transform and the fact that $(\tilde\zeta \chr_{K'})^{1/2} \in L^2(\lie{z}^*)$. We have thus proved that $\sigma_{\tilde\zeta}$ is finite on compacta; by \cite[Theorem 2.18]{rudin_real_1974}, this means that $\sigma_{\tilde\zeta}$ is a regular Borel measure on $\R^n$.

(ii) Without loss of generality, we may take the basis $T_1,\dots,T_d$ of $\lie{z}$ as composed by $\delta_t$-homogeneous elements; thus \eqref{eq:centralsystem} is a homogeneous system, and the associated dilations $\epsilon'_t$ on $\R^n \times \lie{z}^*$ are given by $\epsilon'_t(\lambda,\tau) = (\epsilon_t(\lambda), \tau \circ \delta_{t})$. By \eqref{eq:homogeneity}, $\sigma'$ is $\epsilon'_t$-homogeneous of degree $Q_\delta$. Therefore, if $\tilde\zeta$ is homogeneous of degree $a$, then clearly the measure \eqref{eq:measurenearlyproduct} is homogeneous of degree $Q_\delta+a$; since the canonical projection $\R^n \times \lie{z}^* \to \R^n$ intertwines the two system of dilations, we infer that also $\sigma_{\tilde\zeta}$ is homogeneous of degree $Q_\delta+a$.
\end{proof}

Via interpolation, we then obtain an improvement of Theorem~\ref{thm:weightedestimates}, where the role of the Plancherel measure $\sigma$ is now played by some $\sigma_{\tilde\zeta}$.

\begin{prp}\label{prp:partialweight2}
(i) Suppose that, for some nonnegative Borel functions $w_j : \lie{g}/\lie{z} \to \R$ and $\zeta_j : \lie{z}^* \to \R$ ($j=1,\dots,h$), we have
\begin{equation}\label{eq:Jinequality}
|J(\bar x,\tau)| \geq w_j(\bar x) \, \zeta_j(\tau) \qquad \text{for all $\bar x \in \lie{g}/\lie{z}$, $\tau \in \lie{z}^*$,}
\end{equation}
and set $\tilde w_j(x) = 1 + w_j(P(x))$. Suppose moreover that, for some $\gamma_1,\dots,\gamma_h > 0$, if $\tilde\zeta_{\vec{\gamma}} = \prod_{j=1}^h (1+\zeta_j^{-2\gamma_j})$, then the measure $\sigma_{\tilde\zeta_{\vec{\gamma}}}$ is locally $d$-bounded on some open $\Omega \subseteq \R^n$. If $K \subseteq \Omega$ is compact, $q \in [1,\infty]$, $\alpha \geq 0$,
\[\beta > \alpha + \frac{n}{q}-\frac{d}{\max\{2,q\}},\]
then, for all $m \in \D(\R^n)$ with $\supp m \subseteq K$,
\[\| \langle \cdot \rangle_G^\alpha \, \tilde w_1^{\gamma_1} \cdots \tilde w_h^{\gamma_h} \, \breve m\|_{L^2(G)} \leq C_{K,\alpha,\vec{\gamma},\beta} \|m\|_{B^\beta_{q,q}(\R^n)}.\]

(ii) Suppose in addition that
$\int_{G} \langle x \rangle_G^{-2\alpha} \, \prod_{j=1}^h \tilde w_j(x)^{-2\gamma_j} \,dx < \infty$
for $\alpha > \bar\alpha_{\vec{\gamma}}$. If $K \subseteq \Omega$ is compact, $q \in [1,\infty]$, $\alpha \geq 0$,
\[\beta > \alpha + \bar\alpha_{\vec{\gamma}} + \frac{n}{q}-\frac{d}{\max\{2,q\}},\]
then, for all $m \in \D(\R^n)$ with $\supp m \subseteq K$,
\[\|\langle \cdot \rangle_G^\alpha \, \breve m\|_{L^1(G)} \leq C_{K,\alpha,\vec{\gamma},\beta} \|m\|_{B^\beta_{q,q}(\R^n)}.\]
\end{prp}
\begin{proof}
(i) Since $\sigma_{\tilde\zeta_{\vec{\gamma}}}$ is locally $d$-bounded, the function $\zeta_j$ cannot be everywhere null, therefore \eqref{eq:Jinequality} and the bilinearity of $J$ imply $w_j(P(x)) \leq C \langle x \rangle_G^\theta$ for some $C,\theta \geq 0$,
thus also $\prod_{j=1}^h \tilde w_j(x)^{2\gamma_j} \leq C_{\vec{\gamma}} \langle x \rangle_G^{2\theta(\gamma_1 + \dots + \gamma_h)}$ for some $C_{\vec{\gamma}} \geq 0$.

Let $\psi \in \D(\R^n)$ such that $\psi|_K = 1$ and $K' = \supp \psi \subseteq \Omega$. The operator $m \mapsto \Kern_L (m \psi)$
is then continuous
\[B^\beta_{q,q}(\R^n) \to L^2(G, \langle x \rangle_G^{2\alpha} \textstyle\prod_{j=1}^h \tilde w_j(x)^{2\gamma_j} \,dx)\]
for $\alpha \geq 0$, $\beta > \alpha + \theta(\gamma_1 + \dots \gamma_h) + n$ by Theorem~\ref{thm:weightedestimates}, whereas it is continuous
\[B^\beta_{q,q}(\R^n) \to L^2(G, \textstyle\prod_{j=1}^h \tilde w_j(x)^{2\gamma_j} \,dx)\]
for $\beta > n/q - d/\max\{2,q\}$ by Corollary~\ref{cor:partialweight} and Lemma~\ref{lem:triebeltrace} (cf.\ the proof of Theorem~\ref{thm:weightedestimates}). The conclusion then follows by interpolation.

(ii) It follows from (i) by H\"older's inequality.
\end{proof}

The hypotheses of the previous proposition are quite involved, and it is not particularly clear which classes of groups and systems of operators satisfy them. Hebisch and Zienkiewicz \cite{hebisch_multiplier_1995} treat explicitly the case of direct products of H-type groups; however, as it is mentioned in a remark at the end of \cite{hebisch_multiplier_1995}, there are further cases of homogeneous groups for which this technique gives an improvement of the weighted $L^1$ estimates. In order to attempt a systematic treatment of these various cases, we introduce the following definition: for $h \in \N$, we say that a homogeneous Lie group $G$ is \emph{$h$-capacious} if there exist linearly independent homogeneous elements $\omega_1,\dots,\omega_h \in (\lie{g}/\lie{z})^*$ and linearly independent homogeneous elements $z_1,\dots,z_h \in \lie{z}$ such that, for $j=1,\dots,h$,
\begin{equation}\label{eq:capacious}
|J(\bar x,\tau)| \geq |\omega_j(x)| |\tau(z_j)| \qquad\text{for all $\bar x \in \lie{g}/\lie{z}$, $\tau \in \lie{z}^*$.}
\end{equation}
Clearly, every homogeneous group is $0$-capacious. In the following proposition, we give some criteria which may be of some use in showing that a certain homogeneous group is $h$-capacious. Let us denote by
\begin{equation}\label{eq:descendingcentral}
\lie{g}_{[1]} = \lie{g}, \qquad \lie{g}_{[r+1]} = [\lie{g},\lie{g}_{[r]}]
\end{equation}
the descending central series of a Lie algebra $\lie{g}$.

\begin{prp}\label{prp:capacitycriteria}
Let $G$ be a homogeneous group, with dilations $\delta_t$.
\begin{itemize}
\item[(i)] If $G$ is a M\'etivier group (with any family of automorphic dilations), then $\lie{z} = [\lie{g},\lie{g}]$ and $G$ is $(\dim \lie{z})$-capacious.
\item[(ii)] Suppose that, for some $r \geq 2$, $\dim \lie{g}_{[r]} = 1$. Then $G$ is $1$-capacious.
\item[(iii)] If $\lie{g}$ admits a $\C$-linear structure which is compatible with its homogeneous Lie algebra structure, and if moreover $\dim_\C \lie{g}_{[r]} = 1$ for some $r \geq 2$, then $\lie{g}$ is $2$-capacious.
\item[(iv)] Suppose that $G = G_1 \times G_2$, where $G_1$ and $G_2$ are homogeneous Lie groups with dilations $\delta_{1,t}$ and $\delta_{2,t}$ respectively, so that $\delta_t = \delta_{1,t} \times \delta_{2,t}$. If $G_1$ is $h_1$-capacious and $G_2$ is $h_2$-capacious, then $G$ is $(h_1+h_2)$-capacious.
\end{itemize}
\end{prp}
\begin{proof}
(i) Notice that the capacity map $J$ takes its values in the subspace of $\lie{y}^*$ corresponding to $(\lie{y}/\lie{z})^*$. The condition \eqref{eq:metivier} implies that $J(\cdot,\tau) : \lie{g}/\lie{z} \to (\lie{y}/\lie{z})^*$ is injective for $\tau \neq 0$, and that $J(\bar x,\cdot) : \lie{z}^* \to (\lie{y}/\lie{z})^*$ is injective for $\bar x \neq 0$. Therefore $\dim \lie{g} \leq \dim \lie{y}$, so that $\lie{g} = \lie{y}$ and $[\lie{g},\lie{g}] \subseteq \lie{z}$; on the other hand, $\dim (\lie{g}/\lie{z}) \geq \dim \lie{z}$.

The $\delta_t$ are automorphisms, hence $\lie{z}$ is a homogeneous ideal. Thus, if $h = \dim \lie{z}$, we can choose linearly independent homogeneous elements $z_1,\dots,z_h$ of $\lie{z}$, and also linearly independent homogeneous $\omega_1,\dots,\omega_h \in (\lie{g}/\lie{z})^*$, since $h \leq \dim (\lie{g}/\lie{z})$. Modulo a suitable renormalization, from \eqref{eq:metivier} we infer \eqref{eq:capacious} for $j=1,\dots,h$.

If $[\lie{g},\lie{g}]$ were strictly contained in $\lie{z}$, then we would find $\tau \in \lie{z}^* \setminus \{0\}$ such that $\tau|_{[\lie{g},\lie{g}]} = 0$, but then also $J(\cdot,\tau) = 0$, which contradicts \eqref{eq:metivier}; therefore $\lie{z} = [\lie{g},\lie{g}]$.

(ii) Since $G$ is nilpotent, it must be $r$-step, so that $\lie{g}_{[r]} \subseteq \lie{z}$. Notice that the ideal $\lie{g}_{[r-1]}$ is preserved by every automorphism of $\lie{g}$, therefore it is generated by $\delta_t$-homogeneous elements; since $[\lie{g},\lie{g}_{[r-1]}] = \lie{g}_{[r]} \neq 0$, then there must exist a $\delta_t$-homogeneous element $y \in \lie{g}_{[r-1]}$ such that, for some $x_0 \in \lie{g}$, $[x_0,y] = z \neq 0$. In particular $y \neq 0$ and moreover, since the ideal $\lie{g}_{[r]}$ is $\delta_t$-homogeneous and $1$-dimensional, necessarily $z$ is $\delta_t$-homogeneous.

Since $y \in \lie{g}_{[r-1]}$, the linear map $[\cdot,y] : \lie{g} \to \lie{g}$ takes its values in $\lie{g}_{[r]} = \R z$; therefore, there exists $\omega \in (\lie{g}/\lie{z})^*$ such that $[x,y] = \omega(P(x)) z$ for all $x \in \lie{g}$. Notice that $\omega(P(x_0)) = 1$, thus $\omega \neq 0$; moreover, since both $y$ and $z$ are homogeneous, also $\omega$ is homogeneous. Finally
\begin{equation}\label{eq:linearidentity}
J(\bar x,\tau)(y) = \omega(\bar x) \tau(z) \qquad\text{for all $\bar x \in \lie{g}/\lie{z}$, $\tau \in \lie{z}^*$,}
\end{equation}
which implies immediately that $G$ is $1$-capacious.

(iii) Arguing as in part (ii), but with a complex Lie algebra $\lie{g}$, one finds an identity analogous to \eqref{eq:linearidentity}, where now $\omega$ is a $\C$-linear functional on $\lie{g}/\lie{z}$, and $z \in \lie{z}$. The conclusion then follows by taking the $\R$-linearly independent $\R$-linear functionals $\Re \omega, \Im \omega$ on $\lie{g}/\lie{z}$, and the $\R$-linearly independent elements $z,iz \in \lie{z}$.

(iv) Via the canonical identification $\lie{g} = \lie{g}_1 \times \lie{g}_2$, we have (with the obvious meaning of the notation) $\lie{z} = \lie{z}_1 \times \lie{z}_2$, $\lie{y} = \lie{y}_1 \times \lie{y}_2$, thus also
\[\lie{z}^* = \lie{z}_1^* \times \lie{z}_2^*, \qquad \lie{y}^* = \lie{y}_1^* \times \lie{y}_2^*, \qquad \lie{g}/\lie{z} = (\lie{g}_1/\lie{z}_1) \times (\lie{g}_2 / \lie{z}_2).\]
Moreover clearly $J((\bar x_1,\bar x_2),(\tau_1,\tau_2)) = (J_1(\bar x_1,\tau_1),J_2(\bar x_2,\tau_2))$, therefore
\[|J((\bar x_1,\bar x_2),(\tau_1,\tau_2))| \geq \max \{|J_1(\bar x_1,\tau_1)|,|J_2(\bar x_2,\tau_2)|\}\]
and the conclusion follows immediately.
\end{proof}

Notice that the previous proposition is not sufficient to exhaust all the cases of $h$-capacious groups; an example is shown in \S\ref{subsection:calculations}.

\begin{lem}\label{lem:homogeneousdualbasis}
Suppose that $G$ is $h$-capacious, and let $\omega_1,\dots,\omega_h \in (\lie{g}/\lie{z})^*$ be as in the definition. Then the functionals $\omega_j \circ P$ are null on $[\lie{g},\lie{g}]$. In particular
\[h \leq \min\{\dim \lie{z}, \dim \lie{g} - \dim(\lie{z} + [\lie{g},\lie{g}])\}.\]
Moreover, we can find a homogeneous basis of of $\lie{g}$ compatible with the descending central series such that the functionals $\omega_1 \circ P,\dots,\omega_h \circ P$ are part of the dual basis.
\end{lem}
\begin{proof}
Notice that $[[\lie{g},\lie{g}],\lie{y}] \subseteq [\lie{g},[\lie{g},\lie{y}]] \subseteq [\lie{g},\lie{z}] = 0$. Then from the definition of $J$ it follows that, for every $x \in [\lie{g},\lie{g}]$,
\[J(P(x),\tau) = 0 \qquad\text{for all $\tau \in \lie{z}^*$.}\]
Hence, by choosing in \eqref{eq:capacious} a $\tau \in \lie{z}^*$ such that $\tau(z_j) \neq 0$, we obtain that the functional $\omega_j \circ P$ is null on $[\lie{g},\lie{g}]$. In particular, the $\omega_j \circ P$ correspond to linearly independent elements of 
$(\lie{g}/([\lie{g},\lie{g}] + \lie{z}))^*$, and the inequality about $h$ follows.

Let now $W = \ker (\omega_1 \circ P) \cap \dots \cap \ker (\omega_h \circ P)$. Then $W$ is a homogeneous subspace of $\lie{g}$ containing $[\lie{g},\lie{g}]$. Moreover, if $\tilde\omega_j$ is the element of $(\lie{g}/W)^*$ corresponding to $\omega_j$, then $\tilde\omega_1,\dots,\tilde\omega_h$ are a homogeneous basis of $(\lie{g}/W)^*$. We can then choose homogeneous elements $v_1,\dots,v_h \in \lie{g}$ such that the corresponding elements in the quotient $\lie{g}/W$ are the dual basis of $\tilde\omega_1,\dots,\tilde\omega_h$. Finally, we append to $v_1,\dots,v_h$ a homogeneous basis of $W$ compatible with the descending central series (which, apart from $\lie{g}_{[1]}$, is contained in $W$, and is made of homogeneous ideals), and we are done.
\end{proof}

Here is finally the improvement of Theorem~\ref{thm:weightedestimates} for $h$-capacious groups.

\begin{theo}\label{thm:improvedl1estimates}
Suppose that $G$ is $h$-capacious, and let $Q_G$ be its degree of polynomial growth. Let moreover $L_1,\dots,L_n$ be a homogeneous weighted subcoercive system on $G$. If $q \in [1,\infty]$, $\alpha \geq 0$ and
\[\beta > \alpha + \frac{Q_G - h}{2} + \frac{n}{q} - \frac{1}{\max\{2,q\}},\]
then, for every $K \subseteq \R^n \setminus \{0\}$ compact,
\[\|\breve m\|_{L^1(G,\langle x \rangle_G^\alpha \,dx)} \leq C_{K,\alpha,\beta} \|m\|_{B^\beta_{q,q}}\]
for all $m \in \D(G)$ with $\supp m \subseteq K$.
\end{theo}
\begin{proof}
Let $\omega_1,\dots,\omega_h \in (\lie{g}/\lie{z})^*$ and $z_1,\dots,z_h \in \lie{z}$ be given by the definition of $h$-capacious, and set $w_j(\bar x) = |\omega_j(\bar x)|$, $\zeta_j(\tau) = |\tau(z_j)|$. Notice now that, since the $z_j$ are linearly independent, for every choice of $\gamma_1,\dots,\gamma_h \in \left]0,1/2\right[$, the function $\tilde\zeta_{\vec{\gamma}} = \prod_{j=1}^h (1+\zeta_j^{-2\gamma_j})$ is in $L^1_\loc(\lie{z}^*)$, so that, by Lemma~\ref{lem:homogeneouspushforward}, $\sigma_{\tilde\zeta_{\vec{\gamma}}}$ is a regular Borel measure; in fact, since the $z_j$ are homogeneous, $\sigma_{\tilde\zeta_{\vec{\gamma}}}$ is the sum of $\epsilon_t$-homogeneous regular Borel measures on $\R^n$ (with possibly different degrees of homogeneity), hence $\sigma_{\tilde\zeta_{\vec{\gamma}}}$ is locally $1$-bounded on $\R^n \setminus \{0\}$.

By Lemma~\ref{lem:homogeneousdualbasis}, we can find a homogeneous basis $v_1,\dots,v_k$ of $\lie{g}$, compatible with the descending central series, such that, if $\hat v_1,\dots,\hat v_h$ is the dual basis, then $\hat v_j = \omega_j \circ P$ for $j=1,\dots,h$; in particular we have $\tilde w_j(x) = 1 + |\hat v_j(x)|$. If we set $\kappa_j = \max \{r \tc v_j \in \lie{g}_{[r]}\}$, then $Q_G = \sum_{j=1}^k \kappa_j$ and $\langle x \rangle_G \sim 1 + \sum_{j=1}^k |\hat v_j(x)|^{1/\kappa_j}$ (cf.\ Proposition~2.1 of \cite{martini_spectral}), thus
\[\langle x \rangle_G^{-2\alpha_j} \leq C_{\alpha_j} (1 + |\hat v_j(x)|)^{-2\alpha_j/\kappa_j}\]
for $j=1,\dots,k$ and $\alpha_j \geq 0$. Moreover, since the $\omega_j \circ P$ are null on $[\lie{g},\lie{g}]$, then $\kappa_j = 1$ for $j=1,\dots,h$.

Notice now that, for fixed $\gamma_1,\dots,\gamma_h \in \left]0,1/2\right[$, if $\alpha \geq 0$ satisfies
\[2\alpha > 2\alpha_{\vec{\gamma}} = \sum_{j=1}^h (1-2\gamma_j) + \sum_{j=h+1}^k \kappa_j,\]
then we may choose $\alpha_1,\dots,\alpha_k \geq 0$ such that
\[\alpha = \sum_{j=1}^k \alpha_j, \qquad 2\alpha_j > \begin{cases}
1-2\gamma_j &\text{for $j=1,\dots,h$,}\\
\kappa_j &\text{for $j=h+1,\dots,k$,}
\end{cases}\]
therefore
\[\langle x \rangle_G^{-2\alpha} \, \prod_{j=1}^h \tilde w_j(x)^{-2\gamma_j} \leq \prod_{j=1}^h (1+|\hat v_j(x)|)^{-2(\alpha_j +\gamma_j)} \prod_{j=h+1}^k (1+|\hat v_j(x)|)^{-2\alpha_j/\kappa_j},\]
and the right-hand side is clearly integrable over $G$. We can thus apply Proposition~\ref{prp:partialweight2}(ii), and the conclusion follows because, if the $\gamma_j$ tend to $1/2$, then $\alpha_{\vec{\gamma}}$ tends to $\sum_{j=h+1}^k \kappa_j = Q_G - h$.
\end{proof}

Notice that, if $G$ is a M\'etivier group, by Proposition~\ref{prp:capacitycriteria}(i) we can take $h = \dim[\lie{g},\lie{g}]$, so that $Q_G - h = \dim G$; in fact, by Proposition~\ref{prp:capacitycriteria}(iv), the same holds if $G$ is a direct product of M\'etivier and Euclidean groups.

\section{Mihlin-H\"ormander multipliers}\label{section:mihlin}

Let $G$ be a homogeneous Lie group, with dimension at infinity $Q_G$, automorphic dilations $\delta_t$ and homogeneous dimension $Q_\delta$; as in \cite[\S2.1]{martini_spectral}, we suppose that the homogeneity degrees of the elements of the Lie algebra $\lie{g}$ are not less than $1$, so that $Q_\delta \geq Q_G$. Define $|\cdot|_G$, $\langle \cdot \rangle_G$ as in \S\ref{section:weightedestimates}, and denote by $|\cdot|_\delta$ a subadditive homogeneous norm on $G$ (cf.\ \cite{hebisch_smooth_1990}).

Let $L_1,\dots,L_n$ be a homogeneous weighted subcoercive system on $G$, with associated dilations $\epsilon_t$, and Plancherel measure $\sigma$. Denote moreover by $|\cdot|_\epsilon$ an $\epsilon_t$-homogeneous norm on $\R^n$, smooth off the origin.

Our starting point is, for some $q \in \left[1,\infty\right]$ and $s \in \R$, the following
\begin{quote} \noindent{\bf hypothesis \HP{q}{s}}: 
for some compact $K_0 \subseteq \R^n \setminus \{0\}$ such that
\[\bigcup_{t > 0} \epsilon_t(\mathring{K_0}) = \R^n \setminus \{0\},\]
for all $\beta > s$ and for all $m \in \D(\R^n)$ with $\supp m \subseteq K_0$, we have
\[\|\breve m\|_{L^1(G)} \leq C_{\beta} \|m\|_{B_{q,q}^\beta(\R^n)}.\]
\end{quote}
From Theorems~\ref{thm:weightedestimates} and \ref{thm:improvedl1estimates} we deduce immediately

\begin{prp}\label{prp:HPsatisfaction}
For every $q \in [1,\infty]$, the hypothesis \HP{q}{s} holds in each of the following cases:
\begin{itemize}
\item $s = Q_G/2 + n/q - 1/\max\{2,q\}$;
\item the Plancherel measure $\sigma$ is locally $d$-bounded on $\R^n \setminus \{0\}$ and $s = Q_G/2+ n/q-d/\max\{2,q\}$;
\item $G$ is $h$-capacious and $s = (Q_G-h)/2+n/q-1/\max\{2,q\}$.
\end{itemize}
\end{prp}

In the rest of this section, we forget how such hypothesis may be checked, and we focus on its consequences.

\begin{prp}\label{prp:interpolatedweightedestimates}
Suppose that \HP{q}{s} holds for some $q \in [1,\infty]$ and $s \in \R$. Then $s \geq n/q$. Moreover, for every compact $K \subseteq \R^n \setminus \{0\}$, for every $\alpha \geq 0$ and $\beta > \alpha + s$, for every $D \in \Diff(G)$, for every $m \in B_{q,q}^\beta(\R^n)$ with $\supp m \subseteq K$, we have
\begin{equation}\label{eq:interpolatedweightedestimates}
\|\langle \cdot \rangle_G^\alpha \, D \breve m\|_{L^1(G)} \leq C_{K,D,\alpha,\beta} \|m\|_{B_{q,q}^\beta(\R^n)}.
\end{equation}
\end{prp}
\begin{proof}
Let $\lambda \in \mathring{K}_0 \cap \supp\sigma$. For every $m \in \D(\R^n)$ with $\supp m \subseteq \mathring{K}_0$, we have $|m(\lambda)| \leq \|m\|_{L^\infty(\sigma)} \leq \|\breve m\|_1 \leq C_\beta \|m\|_{B_{q,q}^\beta}$ for all $\beta > s$. Such an inequality gives easily $\|f\|_\infty \leq C_\beta \|f\|_{B_{q,q}^\beta}$ for all $f$ in the Schwartz class and $\beta > s$, which however can hold only if $s \geq n/q$ (cf.\ Theorem 1 of \S2.6.2 of \cite{triebel_spaces_1978}).

The hypothesis \HP{q}{s} gives \eqref{eq:interpolatedweightedestimates} in the case $\alpha = 0$, $D = 1$, $m$ smooth, $K = K_0$. The extension to a generic compact $K \subseteq \R^n \setminus \{0\}$ is performed by a partition-of-unity argument and exploiting homogeneity. The full generality is then reached by approximation and interpolation, as in the proof of Theorem~\ref{thm:weightedestimates}.
\end{proof}

Notice that, by \cite[Proposition 2.1]{martini_spectral}, there are constants $a,C > 0$ such that
\begin{equation}\label{eq:moduluscomparison}
1 + |x|_\delta \leq C \langle x \rangle_G^a.
\end{equation}

\begin{cor}\label{cor:usefulestimates}
Suppose that \HP{q}{s} holds. Let $K \subseteq \R^n \setminus \{0\}$ be compact, $\beta > s$. If $m \in B^\beta_{q,q}(\R^n)$ and $\supp m \subseteq K$, then $\breve m \in L^1(G)$. Moreover, for $0 \leq \alpha < (\beta - s)/a$,
\begin{equation}\label{eq:weightedl1}
\int_G (1+|x|_\delta)^\alpha |\breve m(x)| \,dx \leq C_{K,\alpha,\beta} \|m\|_{B_{q,q}^\beta}
\end{equation}
and, for all $h \in G$,
\begin{equation}\label{eq:lipschitz}
\| \RA_h \breve m - \breve m\|_1 \leq C_{K,\beta} \|m\|_{B_{q,q}^\beta} |h|_\delta, \qquad 
\|\LA_h \breve m - \breve m\|_1 \leq C_{K,\beta} \|m\|_{B_{q,q}^\beta} |h|_\delta.
\end{equation}
\end{cor}
\begin{proof}
Since $\beta > s + a\alpha$, by Proposition~\ref{prp:interpolatedweightedestimates} and \eqref{eq:moduluscomparison} we have
\[\int_G (1+|x|_\delta)^\alpha |\breve m(x)| \,dx \leq C_\alpha \int_G \langle x \rangle_G^{a\alpha} |\breve m(x)| \,dx \leq C_{K,\alpha,\beta} \|m\|_{B_{q,q}^\beta},\]
and in particular $\breve m \in L^1(G)$.

Starting from the inequality $\| \RA_{\exp(tX)} \breve m - \breve m\|_1 \leq \|X \breve m\|_1 |t|$, true for all $X \in \lie{g}$ and $t \in \R$, having chosen a basis $X_1,\dots,X_k$ of $\lie{g}$, with $X_j$ homogeneous of degree $d_j$, we easily obtain $\| \RA_h \breve m - \breve m \|_1 \leq C \sum_{j=1}^k \|X_j \breve m\|_1 |h|_\delta^{d_j}$, so that also
\[\| \RA_h \breve m - \breve m \|_1 \leq C \left(\|\breve m\|_1 + \textstyle\sum_{j=1}^k \|X_j \breve m\|_1\right) |h|_\delta,\]
since $d_j \geq 1$. However $\|\breve m\|_1 + \sum_{j=1}^k \|X_j \breve m\|_1 \leq C_{K,\beta} \|m\|_{B^\beta_{q,q}}$ by Proposition~\ref{prp:interpolatedweightedestimates}, thus we get the former of \eqref{eq:lipschitz}; the latter is obtained by replacing $m$ with $\overline{m}$.
\end{proof}

\begin{lem}\label{lem:spectraldecomposition}
Let $m$ be a bounded Borel function on $\R^n$. Then we can find bounded Borel functions $m_j$ on $\R^n$ (for $j \in \Z$) such that
\begin{equation}\label{eq:supportdecomposition}
\supp m_j \subseteq \{\lambda \tc 2^{-1} \leq |\lambda|_\epsilon \leq 2\}, \qquad \|m_j\|_{B_{q,q}^\beta} \leq C_{q,\beta} \|m\|_{M_\epsilon B_{q,q}^\beta}
\end{equation}
for all $q \in [1,\infty]$ and $\beta \geq 0$, and moreover
\begin{equation}\label{eq:spectraldecomposition}
\breve m = \sum_{j \in \Z} 2^{-Q_\delta j} \breve m_j \circ \delta_{2^{-j}},
\end{equation}
in the sense of strong convergence of the corresponding convolution operators.
\end{lem}
\begin{proof}
Set $K = \{\lambda \tc 2^{-1} \leq |\lambda|_\epsilon \leq 2\}$. Choose a nonnegative $\eta \in \D(\R^n)$ supported in $K$ and such that $\sum_{j \in \Z} \eta \circ \epsilon_{2^j} = 1$ off the origin, and let $m_j =  (m \circ \epsilon_{2^{-j}}) \, \eta$. Then clearly \eqref{eq:supportdecomposition} is satisfied, and moreover $m = \sum_{j \in \Z} m_j \circ \epsilon_{2^j}$ off the origin. In fact, this is locally a finite sum and the convergence is dominated by the constant $\|m\|_\infty$. Since $\sigma(\{0\}) = 0$, by the spectral theorem and \eqref{eq:homogeneity} we then have \eqref{eq:spectraldecomposition}, in the sense of strong convergence of the corresponding convolution operators.
\end{proof}

\begin{prp}\label{prp:singularkernel}
Suppose that \HP{q}{s} holds. Let $\beta > s$. If $m$ is a bounded Borel function on $\R^n$ such that $\|m\|_{M_\epsilon B_{q,q}^\beta} < \infty$, then $\breve m|_{G \setminus \{e\}} \in L^1_\loc(G \setminus \{e\})$, and moreover
\begin{gather}
\label{eq:rightcz} \int_{|x|_\delta \geq 2|h|_\delta} |\breve m(xh) - \breve m(x)| \,dx \leq C_{\beta} \|m\|_{M_\epsilon B_{q,q}^\beta},\\
\label{eq:leftcz} \int_{|x|_\delta \geq 2|h|_\delta} |\breve m(hx) - \breve m(x)| \,dx \leq C_{\beta} \|m\|_{M_\epsilon B_{q,q}^\beta}
\end{gather}
for all $h \in G \setminus \{e\}$.
\end{prp}
\begin{proof}
Let the $m_j$ be given by Lemma~\ref{lem:spectraldecomposition} and set $u_j = 2^{-Q_\delta j} \breve m_j \circ \delta_{2^{-j}}$.

Firstly we prove that the convergence in \eqref{eq:spectraldecomposition} holds also in $L^1_\loc(G \setminus \{e\})$. In fact, let $B_k = \{ x \in G \tc 2^k \leq |x|_\delta \leq 2^{k+1}\}$; it is sufficient to prove the convergence in each $L^1(B_k)$. We have $\int_{B_k} |u_j| \,d\mu = \int_{B_{k-j}} |\breve m_j| \,d\mu$ and, for $j \leq k$,
\[\int_{B_{k-j}} |\breve m_j(x)| \,dx \leq 2^{\alpha (j-k)} \int_{B_{k-j}} |\breve m_j(x)| |x|_\delta^{\alpha} \,dx \leq C 2^{\alpha(j-k)} \|m\|_{M_\epsilon B_{q,q}^\beta}\]
(where $\alpha > 0$ is as in \eqref{eq:weightedl1}), whereas, for $j \geq k$,
\[\int_{B_{k-j}} |\breve m_j(x)| \,dx \leq \|\breve m_j\|_2  \sqrt{\mu(B_{k-j})} \leq \sqrt{\sigma(K)\mu(B_0)} \|m\|_\infty \,  2^{Q_\delta (k-j)/2},\]
(here we use a uniform estimate on the $L^2$-norms of the $\breve m_j$) so that
\[\sum_j \int_{B_k} |u_j| \,d\mu \leq C' \sum_{j \leq k} 2^{\alpha(j-k)} + C'' \sum_{j \geq k} 2^{Q_\delta (k-j)/2} < \infty.\]
This shows (by uniqueness of limits) that the restriction of the distribution $\breve m$ to $G \setminus \{e\}$ coincides with a function in $L^1_\loc(G \setminus \{e\})$.

Since $\breve m = \sum_{j \in \Z} u_j$ in $L^1_\loc(G \setminus \{e\})$, then $\RA_h \breve m - \breve m = \sum_{j \in \Z} (\RA_h u_j - u_j)$ in $L^1_\loc(G \setminus \{e,h^{-1}\})$, so that in particular
\begin{equation}\label{eq:controlsum}
\int_{|x|_\delta \geq 2|h|_\delta} |\breve m(xh) - \breve m(x)| \,dx \leq \sum_{j \in \Z} \int_{|x|_\delta \geq 2|h|_\delta} |u_j(xh) - u_j(x)| \,dx.
\end{equation}
Let $k \in \Z$. Then, for $j < k$, the $j$-th summand in the right-hand side of \eqref{eq:controlsum} is not greater than
\[
2 \int_{|x|_\delta \geq |h|_\delta} |u_j(x)| \,dx \leq C_{\beta} \frac{2^{\alpha j}}{|h|_\delta^\alpha} \|m\|_{M_\epsilon B_{q,q}^\beta}\]
by \eqref{eq:weightedl1}, whereas, for $j \geq k$, it is not greater than
\[
\int_G |\breve m_j(y \delta_{2^{-j}}(h)) - \breve m_j(y)| \,dy 
\leq C_{\beta} \frac{|h|_\delta}{2^j} \|m\|_{M_\epsilon B_{q,q}^\beta}
\]
by \eqref{eq:lipschitz}. Putting all together, the left-hand side of \eqref{eq:controlsum} is majorized by
\[
C_{\beta} \|m\|_{M_\epsilon B_{q,q}^\beta} \left(\frac{2^{k\alpha}}{|h|_\delta^\alpha} \sum_{j<0} 2^{j\alpha} + \frac{|h|_\delta}{2^{k}} \sum_{j \geq 0} 2^{-j}\right)\]
and, in order to obtain an estimate independent of $h$, it is sufficient to choose a $k$ such that $2^k \leq |h|_\delta < 2^{k+1}$. Hence we have proved \eqref{eq:rightcz}; the inequality \eqref{eq:leftcz} is obtained analogously.
\end{proof}

Here is finally the multiplier theorem.

\begin{theo}\label{thm:mihlinhoermander}
Suppose that \HP{q}{s} holds. If $m$ is a bounded Borel function on $\R^n$ such that $\|m\|_{M_\epsilon B_{q,q}^\beta} < \infty$ for some $\beta > s$, then the operator $m(L)$ is of weak type $(1,1)$ and bounded on $L^p(G)$ for $1 < p < \infty$, with
\[\|m(L)\|_{p \to p} \leq C_{p,q,\beta} \|m\|_{M_\epsilon B_{q,q}^\beta}.\]
\end{theo}
\begin{proof}
Notice that $\|m\|_{L^\infty(\sigma)} \leq C_{q,\beta} \|m\|_{M_\epsilon B_{q,q}^\beta}$, since $\beta > n/p$ by Proposition~\ref{prp:interpolatedweightedestimates}. In view of Proposition~\ref{prp:singularkernel}, the conclusion then follows from the Calder\'on-Zygmund theory of singular integral operators \cite[\S I.5, Theorem 3 and \S I.7.4(iii)]{stein_harmonic_1993}.
\end{proof}

Notice that a compactly supported $m \in B_{q,q}^\beta(\R^n)$ does satisfy an $L^q$ Mihlin-H\"ormander condition of order $\beta$, at least for $\beta$ sufficiently large. More precisely, let $\tilde Q_\epsilon = \sum_j w_j/ \min_j w_j$ denote the \emph{normalized homogeneous dimension} associated with the dilations $\epsilon_t(\lambda) = (t^{w_1} \lambda_1,\dots,t^{w_n} \lambda_n)$; then we have

\begin{prp}\label{prp:mhcompactsupp}
If $K \subseteq \R^n$ is compact, $q \in [1,\infty]$, $\beta > \tilde Q_\epsilon/q$, then
\[\|m\|_{M_\epsilon B_{q,q}^\beta} \leq C_{K,q,\beta} \|m\|_{B_{q,q}^\beta}\]
for all $m \in B_{q,q}^\beta(\R^n)$ with $\supp m \subseteq K$.
\end{prp}

Recall that $\|m\|_{M_\epsilon B_{q,q}^\beta} = \sup_{t > 0} \| (m \circ \epsilon_t) \,\eta \|_{B_{q,q}^\beta}$ for a suitable cut-off function $\eta \in \D(\R^n)$ supported away from the origin. If $\supp m \subseteq K$, then, for some $t_K > 0$ sufficiently large, we have $\|(m \circ \epsilon_t) \,\eta\|_{B_{q,q}^\beta} = 0$ for $t > t_K$; therefore Proposition~\ref{prp:mhcompactsupp} follows immediately from

\begin{lem}\label{lem:smalldilations}
If $p,q \in [1,\infty]$, $\beta > \tilde Q_\epsilon/p$ and $\eta \in \D(\R^n)$, then
\[\sup_{0 < t \leq 1} \| (f \circ \epsilon_t) \, \eta\|_{B_{p,q}^\beta} \leq C_{\eta,p,q,\beta} \|f\|_{B_{p,q}^\beta}.\]
\end{lem}
\begin{proof}
Without loss of generality, we may suppose that $\beta$ is not an integer (the missing values can be recovered a posteriori by interpolation), thus there exists $m \in \N$ such that $m-1 < \beta < m$.

Define the $k$-th order difference operator $\Delta^k_y$ and the modulus of continuity $\omega^k_p$ as in \cite[\S 6.2]{bergh_interpolation_1976}. In view of the characterization by differences of the Besov norm given in \cite[Theorem 6.2.5]{bergh_interpolation_1976}, we have to estimate
\begin{equation}\label{eq:differencesnorm}
\| (f \circ \epsilon_t) \, \eta \|_p + \left( \int_0^1 \left(\frac{\omega_p^m(r, (f \circ \epsilon_t) \eta)}{r^\beta}\right)^q \frac{dr}{r} \right)^{1/q}.
\end{equation}

The former summand in \eqref{eq:differencesnorm} is immediately majorized by H\"older's inequality and embeddings, since $\eta$ is compactly supported and $\beta > n/p$:
\[ \| (f \circ \epsilon_t) \, \eta \|_p \leq C_{\eta,p} \| f \circ \epsilon_t \|_\infty = C_{\eta,p} \|f\|_\infty \leq C_{\eta,p,q,\beta} \|f\|_{B_{p,q}^\beta}.\]

For the latter summand, notice first that
\begin{equation}\label{eq:valormedio}
\|\Delta^k_y \psi\|_p \leq C_{k,p} \|\psi\|_{W^k_p} |y|^k;
\end{equation}
this inequality, together with the Leibniz rule for finite differences, H\"older's inequality and the fact that $\eta \in \D(\R^n)$, gives easily
\[ \| \Delta_y^m ((f \circ \epsilon_t) \, \eta) \|_p \leq C_{\eta,m,p,p_0,\dots,p_m} \sum_{k=0}^m |y|^{m-k} \| \Delta_y^k (f \circ \epsilon_t) \|_{p_k}\]
for any choice of $p_0,\dots,p_k \geq p$; since
\[ \| \Delta_y^k (f \circ \epsilon_t) \|_{p_k} = t^{-Q_\epsilon/p_k}  \| \Delta_{\epsilon_t(y)}^k f \|_{p_k} \qquad\text{and}\qquad |\epsilon_t(y)|_\infty \leq t^{w_*} |y|_\infty\]
for $t \leq 1$, where $w_* = \min\{w_1,\dots,w_n\}$, we then get also
\[ \omega_p^m(r,(f \circ \epsilon_t) \eta) \leq C_{\eta,m,p,p_0,\dots,p_m} \sum_{k=0}^m r^{m-k} t^{-Q_\epsilon/p_k} \omega_{p_k}^k (t^{w_*} r, f).\]
Choose now $p_m = p$, and $p_k = p \beta/k$ for $k < m$. Then, for $k < m$, we have
\[p_k > p, \quad k - \frac{n}{p_k} = \frac{k}{\beta} \left( \beta - \frac{n}{p} \right) < \beta - \frac{n}{p}, \quad  w_* k - \frac{Q_\epsilon}{p_k} = \frac{w_* k}{\beta} \left( \beta - \frac{\tilde Q_\epsilon}{p} \right) > 0,\]
so that, by \eqref{eq:valormedio} and the embeddings $B_{p,q}^\beta \subseteq B_{p_k,1}^k \subseteq W^{k}_{p_k}$,
\[\frac{r^{m-k} t^{-Q_\epsilon/p_k} \omega_{p_k}^k (t^{w_*} r, f)}{r^\beta} \leq C_{p,q,\beta} r^{m-\beta} \|f\|_{B_{p,q}^\beta}.\]
For $k=m$, instead,
\[\frac{t^{-Q_\epsilon/p} \omega_{p}^m (t^{w_*} r, f)}{r^\beta} = t^{w_* \beta -Q_\epsilon/p} \frac{\omega_{p}^m (t^{w_*} r, f)}{(t^{w_*}r)^\beta} \leq \frac{\omega_{p}^m (t^{w_*} r, f)}{(t^{w_*}r)^\beta}.\]
Putting all together, the latter summand in \eqref{eq:differencesnorm} is majorized by
\[ C_{\eta,p,q,\beta} \left( \|f\|_{B_{p,q}^\beta} \left(\int_0^1 (r^{m-\beta})^q \frac{dr}{r}\right)^{1/q} + \left(\int_0^{w_*} \left(\frac{\omega_{p}^m (r, f)}{r^\beta}\right)^q \frac{dr}{r} \right)^{1/q} \right),\]
and the conclusion follows again by \cite[Theorem 6.2.5]{bergh_interpolation_1976}.
\end{proof}

\section{Marcinkiewicz multipliers}\label{section:marcinkiewicz}

Let $G$ be a homogeneous Lie group, with automorphic dilations $\delta_t$ and homogeneous dimension $Q_\delta$. For $w \in L^1(G)$, we define the \emph{maximal operator} $M_w \phi(x) = \sup_{t > 0} |\phi * (t^{-Q_\delta} w \circ \delta_{t^{-1}})(x)|$ 
.
We say that the function $w$ is \emph{M-admissible}\index{function!M-admissible} if $M_w$ is bounded on $L^p(G)$ for $1 < p < \infty$.

In terms of maximal operators, we formulate the following hypothesis about the homogeneous group $G$ and a chosen homogeneous weighted subcoercive system $L_1,\dots,L_n$ on it:
\begin{quote}
{\bf hypothesis \HPK{s}{d}:} for every $\beta > s$ there exist
\begin{itemize}
\item a Borel function $u_\beta$ on $G$ with $u_\beta = u_\beta^*$ and $u_\beta \geq c \langle \cdot \rangle_G^{-\theta}$ for some $c,\theta > 0$,
\item a positive regular Borel measure $\sigma_\beta$ on $\R^n$, which is locally $d$-bounded on $\R^n \setminus \{0\}$,
\item a non-negative real number $\gamma_\beta < 2\beta$,
\end{itemize}
such that
\begin{itemize}
\item the function $\langle \cdot \rangle_G^{-\gamma_\beta} u_\beta$ is M-admissible, and
\item for every compact $K \subseteq \R^n \setminus \{0\}$ and every $m \in \D(\R^n)$ with $\supp m \subseteq K$, we have
\begin{equation}\label{eq:quasiplancherel}
\|\breve m\|_{L^2(G, u_\beta^{-1}(x) \,dx)} \leq C_{K,\beta} \|m\|_{L^2(\sigma_\beta)}.
\end{equation}
\end{itemize}
\end{quote}


\begin{prp}\label{prp:HPKsatisfaction}
Let $G$ be a homogeneous group, with degree of polynomial growth $Q_G$, and let $L_1,\dots,L_n$ be a homogeneous weighted subcoercive system on $G$.
\begin{itemize}
\item[(i)] The hypothesis \HPK{Q_G/2}{1} holds. More generally, if the Plancherel measure $\sigma$ is locally $d$-bounded on $\R^n \setminus \{0\}$, then \HPK{Q_G/2}{d} holds.
\item[(ii)] If $G$ is $h$-capacious, then \HPK{(Q_G-h)/2}{1} holds.
\end{itemize}
\end{prp}
\begin{proof}
(i) Let $\sigma$ be the Plancherel measure associated with the system $L_1,\dots,L_n$. For $\beta > Q_G/2$, we choose $u_\beta \equiv 1$, $\sigma_\beta = \sigma$. By \eqref{eq:homogeneity}, $\sigma$ is $\delta_t$-homogeneous, so that it is locally $1$-bounded on $\R^n \setminus \{0\}$. Therefore, by the Plancherel formula, in order to conclude, it is sufficient to show that, for $\gamma_\beta \in \left]Q_G,2\beta\right[$, the function $w_\beta = \langle \cdot \rangle_G^{-\gamma_\beta}$ is M-admissible.

The ideals composing the descending central series \eqref{eq:descendingcentral} are characteristic and thus $\delta_t$-invariant, hence we can find $\delta_t$-invariant complements $V_k$ of $\lie{g}_{[k+1]}$ in $\lie{g}_{[k]}$. The dilations $\tilde\delta_t$ of $\lie{g}$ defined by $\tilde\delta_t(x) = t^k x$ for $x \in V_k$ in general are not automorphic, but commute with the $\delta_t$, and moreover, by \cite[Proposition 2.1]{martini_spectral}, if $|\cdot|_{\tilde\delta}$ is a $\tilde\delta_t$-homogeneous norm, then $\langle \cdot \rangle_G \sim 1+|\cdot|_{\tilde\delta}$. We then have
\begin{multline*}
\left| \int_G \phi(x \, \delta_t(y)^{-1} ) \, w_\beta(y) \,dy \right| = \left|\int_{|y|_{\tilde\delta} \leq 1} + \sum_{h \geq 1} \int_{2^{h-1} < |y|_{\tilde\delta} \leq 2^h}\right|\\
\leq C_{\gamma_\beta} \sum_{h \geq 0} 2^{-h(\gamma_\beta-Q_G)} \int_{|y|_{\tilde\delta} \leq 1} |\phi(x \,\delta_t(\tilde\delta_{2^h}(y))^{-1})| \,dy.
\end{multline*}
Since $\gamma_\beta > Q_G$, if $M_\strong$ is the strong maximal function on $G$ associated to a basis of simultaneous eigenvectors of the $\delta_t$ and the $\tilde\delta_t$ \cite[\S2]{christ_strong_1992}, we obtain
\[M_{w_\beta} \phi \leq C_{\gamma_\beta} M_{\strong} \phi,\]
which gives the conclusion by \cite[Theorem 2.1]{christ_strong_1992}.

(ii) Let $\lie{z}$ be the center of $\lie{g}$, and $P : \lie{g} \to \lie{g}/\lie{z}$ be the canonical projection. Let $\omega_1,\dots,\omega_h \in (\lie{g}/\lie{z})^*$ and $z_1,\dots,z_h \in \lie{z}$ be given by the definition of $h$-capacious. By Lemma~\ref{lem:homogeneousdualbasis}, there exists a homogeneous basis $v_1,\dots,v_k$ of $\lie{g}$ compatible with the descending central series such that, if $\hat v_1,\dots,\hat v_k$ is the dual basis, then $\omega_j \circ P = \hat v_j$ for $j=1,\dots,h$. Moreover, if we set $\kappa_j = \max \{r : v_j \in \lie{g}_{[r]}\}$, then $\kappa_j = 1$ for $j=1,\dots,h$ and
\begin{equation}\label{eq:nilpotentgrowth}
Q_G = \textstyle\sum_{j=1}^k \kappa_j, \qquad \langle x \rangle_G \sim 1 + \sum_{j=1}^k |\hat v_j(x)|^{1/\kappa_j}
\end{equation}
by \cite[Proposition 2.1]{martini_spectral}.

Let $\daleth_{\vec{t}}$ be the $k$-parameter family of dilations on $\lie{g}$ given by $\daleth_{\vec{t}}(v_j) = t_j v_j$. Clearly the $\daleth_{\vec{t}}$ are in general not automorphisms, but the automorphic dilations $\delta_t$ can be obtained as a particular case: $\delta_t = \daleth_{(t^{b_1},\dots,t^{b_k})}$, where $b_j$ is the $\delta_t$-homogeneous degree of $v_j$.

If $\beta > (Q_G-h)/2$, then $2\beta > Q_G - h = \sum_{j=h+1}^k \kappa_j$, so that we can find $\eta_{\beta,1},\dots,\eta_{\beta,h} \in \left[0,1\right[$ and $\gamma_{\beta,1},\dots,\gamma_{\beta,k} > 0$ such that
\[2\beta > \gamma_\beta = \sum_{j=1}^k \gamma_{\beta,j}, \qquad \gamma_{\beta,j} > \begin{cases}
1 - \eta_{\beta,j} &\text{for $j=1,\dots,h$,}\\
\kappa_j &\text{for $j=h+1,\dots,k$.}
\end{cases}\]

Let now $\sigma'$ be the Plancherel measure on $\R^n \times \lie{z}^*$ associated to the system $L_1,\dots,L_n$ extended with the central derivatives, as in \S\ref{section:metivier}, and let $\sigma_\beta$ be the push-forward of the measure $\prod_{j=1}^h (1+|\tau(z_j)|^{-\eta_{\beta,j}}) \,d\sigma'(\lambda,\tau)$ via the canonical projection on the first factor of $\R^n \times \lie{z}^*$. By Lemma~\ref{lem:homogeneouspushforward}, since $\eta_{\beta,1},\dots,\eta_{\beta,h} < 1$, the measure $\sigma_{\beta}$ is a regular Borel measure on $\R^n$; moreover, since the $z_j$ are $\delta_t$-homogeneous, $\sigma_{\beta}$ is the sum of $\epsilon_t$-homogeneous regular Borel measures of different degrees (where $\epsilon_t$ are the dilations associated with the system $L_1,\dots,L_n$), and consequently $\sigma_\beta$ is locally $1$-bounded on $\R^n \setminus \{0\}$. Finally, if we set
$u_\beta(x) = \prod_{j=1}^h (1+|\omega_j(P(x))|)^{-\eta_j} = \prod_{j=1}^h (1+|\hat v_j(x)|)^{-\eta_j}$,
then $u_{\beta} = u_{\beta}^*$ and, by \eqref{eq:nilpotentgrowth}, $u_\beta^{-1}$ is dominated by some power of $\langle \cdot \rangle_G$; moreover, by Corollary~\ref{cor:partialweight}, for every compact $K \subseteq \R^n \setminus \{0\}$ and every $m \in \D(\R^n)$ with $\supp m \subseteq K$, we have \eqref{eq:quasiplancherel}.

In order to conclude, we must show that $w_\beta = \langle \cdot \rangle_G^{-\gamma_\beta} u_\beta$ is M-admissible. In fact, again by \eqref{eq:nilpotentgrowth},
\[w_\beta(x) \leq C_\beta \, {\textstyle\prod_{j=1}^h} (1+|\hat v_j(x)|)^{-(\gamma_{\beta,j} + \eta_{\beta,j})} \, {\textstyle\prod_{j=h+1}^k} (1+|\hat v_j(x)|)^{-\gamma_{\beta,j}/\kappa_j},\]
and the exponents $\gamma_{\beta,j} + \eta_{\beta,j}$, $\gamma_{\beta,j}/\kappa_j$ are all greater than $1$ by construction. The conclusion then follows as in part (i), but with a multi-variate decomposition, by \cite[Theorem 2.1]{christ_strong_1992} applied to the multi-parameter dilations $\daleth_{\vec t}$.
\end{proof}

Suppose now that, for $l=1,\dots,\ell$, $G_l$ is a homogeneous Lie group, with dilations $(\delta_{l,t})_{t > 0}$, and that $L_{l,1},\dots,L_{l,n_l}$ is a homogeneous weighted subcoercive system on $G_l$. Set $G^\times = G_1 \times \dots \times G_\ell$,
and let $L^\times_{l,j} \in \Diff(G^\times)$ be defined by
\[L^\times_{l,j}(f_1 \otimes \dots \otimes f_\ell) = f_1 \otimes \dots \otimes (L_{l,j} f_j) \otimes \dots \otimes f_\ell\]
for $l=1,\dots,\ell$, $j=1,\dots,n_l$. By \cite[Corollary 5.5]{martini_spectral}, we know that
\[L_{1,1}^\times,\dots,L_{1,n_1}^\times,\dots,L_{\ell,1}^\times,\dots,L_{\ell,n_\ell}^\times\]
is a homogeneous weighted subcoercive system on $G^\times$.

We then show how the hypotheses on the factor groups $G_l$ can be put together in order to obtain weighted estimates on the product group $G^\times$. In the following, inequalities involving vectors are to be read componentwise.

\begin{prp}\label{prp:productestimates}
Suppose that, for $l=1,\dots,\ell$, the homogeneous group $G_l$, with the system $L_{l,1},\dots,L_{l,n_l}$, satisfies \HPK{s_l}{d_l}. For $q \in [1,\infty]$, if
\[\vec{\beta} > \vec{s} + \frac{\vec{n}}{q} - \frac{\vec{d}}{\max\{2,q\}},\]
where $\vec{s} = (s_1,\dots,s_\ell)$, $\vec{d} = (d_1,\dots,d_\ell)$, then there exists
\[w_{\vec{\beta}} = w_{\vec{\beta},1} \otimes \dots \otimes w_{\vec{\beta},\ell} \in L^1(G^\times),\]
with $w_{\vec{\beta}} > 0$, $w_{\vec{\beta}}^* = w_{\vec{\beta}}$, such that $w_{\vec{\beta},l}$ is M-admissible on $G_l$ for $l=1,\dots,\ell$, and moreover, for every compact $K = \prod_{l=1}^\ell K_l \subseteq \prod_{l=1}^\ell (\R^{n_l} \setminus \{0\})$, and for every $m \in S^{\vec{\beta}}_{p,q}B(\R^{\vec{n}})$ with $\supp m \subseteq K$, we have
\[\|\Kern_{L^\times} m\|_{L^2(G^\times, w_{\vec{\beta}}^{-1}(x) \,dx)} \leq C_{K,\vec{\beta},p,q} \|m\|_{S^{\vec{\beta}}_{p,q}B(\R^{\vec{n}})}.\]
\end{prp}
\begin{proof}
Take $\vec{\alpha}$ such that $\vec{\alpha} > \vec{s}$, $\vec{\beta} > \vec{\alpha} + \vec{n}/q - \vec{d}/\max\{2,q\}$. For $l=1,\dots,\ell$, since $\alpha_l > s_l$, by \HPK{s_l}{d_l} we can find a function $u_{\vec{\alpha},l} = u_{\vec{\alpha},l}^* > 0$ on $G_l$ such that $u_{\vec{\alpha},l} \geq c_l \langle \cdot \rangle_{G_l}^{-\theta_l}$ for some $c_l,\theta_l > 0$, a positive regular Borel measure $\sigma_{\vec{\alpha},l}$ on $\R^{n_l}$ locally $d_l$-bounded on $\R^{n_l} \setminus \{0\}$, and a positive real number $\gamma_{\vec{\alpha},l} < 2\alpha_l$ such that the function $w_{\vec{\beta},l} = \langle \cdot \rangle_{G_l}^{-\gamma_{\vec{\alpha},l}} u_{\vec{\alpha},l}$ is M-admissible on $G_l$ and
\begin{equation}\label{eq:l2weightmeasure}
\|\Kern_{L_l} m_l\|_{L^2(G_l,u_{\vec{\alpha},l}^{-1}(x_l) \,dx_l)} \leq C_{K_l,\alpha_l} \|m\|_{L^2(\sigma_{\vec{\alpha},l})}
\end{equation}
for every compact $K_l \subseteq \R^{n_l} \setminus \{0\}$ and every $m_l \in \D(\R^{n_l})$ with $\supp m_l \subseteq K_l$.

Set $u_{\vec{\alpha}} = u_{\vec{\alpha},1} \otimes \dots \otimes u_{\vec{\alpha},\ell}$, $\sigma_{\vec{\alpha}} = \sigma_{\vec{\alpha},1} \times \dots \times \sigma_{\vec{\alpha},\ell}$. By ``taking the Hilbert tensor product'' of the inequalities \eqref{eq:l2weightmeasure}, from \cite[Corollary~5.5]{martini_spectral} we deduce that
\[\|\Kern_{L^\times} m\|_{L^2(G^\times,u_{\vec{\alpha}}^{-1}(x) \,dx)} \leq C_{K,\vec{\alpha}} \|m\|_{L^2(\sigma_{\vec{\alpha}})}\]
for every compact $K = \prod_{l=1}^\ell K_l \subseteq \prod_{l=1}^\ell (\R^{n_\ell} \setminus \{0\})$ and every $m \in \D(\R^{\vec{n}})$ with $\supp m \subseteq K$.

Notice now that, again by taking tensor products, Lemma~\ref{lem:triebeltrace} gives
\[\|m\|_{L^2(\sigma_{\vec{\alpha}})} \leq C_{K,\vec{\alpha},\vec{\eta}} \|m\|_{S^{\vec{\eta}}_{2,2}B(\R^{\vec{n}})}\]
for $\vec{\eta} > (\vec{n} - \vec{d})/2$, whereas trivially
\[\|m\|_{L^2(\sigma_{\vec{\alpha}})} \leq C_{K,\vec{\alpha}} \|m\|_\infty \leq C_{K,\vec{\alpha}} \|m\|_{S^{0}_{\infty,1}B(\R^{\vec{n}})},\]
so that, by embeddings and interpolation (cf.\ the proof of Theorem~\ref{thm:weightedestimates}),
\[\|m\|_{L^2(\sigma_{\vec{\alpha}})} \leq C_{K,\vec{\alpha},\vec{\eta},q} \|m\|_{S^{\vec{\eta}}_{q,q}B(\R^{\vec{n}})}\]
for $\vec{\eta} > \vec{n}/q - \vec{d}/\max\{2,q\}$.

Putting all togehter, we have
\[\|\Kern_{L^\times} m\|_{L^2(G^\times,u_{\vec{\alpha}}^{-1}(x) \,dx)} \leq C_{K,\vec{\alpha},\vec{\eta},q} \|m\|_{S^{\vec{\eta}}_{q,q}B(\R^{\vec{n}})}\]
for $\vec{\eta} > \vec{n}/q - \vec{d}/\max\{2,q\}$. On the other hand, by Theorem~\ref{thm:weightedestimates},\[\|\Kern_{L_l} m_l\|_{L^2(G_l,\langle x_l \rangle_{G_l}^{\gamma_l} u_{\vec{\alpha},l}^{-1}(x_l) \,dx_l)} \leq C_{K_l,\vec{\alpha},\gamma_l,\eta_l} \|m_l\|_{B^{\eta_l}_{2,2}(\R^{n_l})}\]
for $\eta_l > \gamma_l/2 + \theta_l/2 + n_l/2$, so that, by tensor products and embeddings,
\[\|\Kern_{L^\times} m\|_{L^2(G^\times,\langle x_1 \rangle_{G_1}^{\gamma_1} \cdots \langle x_{\ell} \rangle_{G_\ell}^{\gamma_\ell} u_{\vec{\alpha}}^{-1}(x) \,dx)} \leq C_{K,\vec{\alpha},\vec{\gamma},\vec{\eta},q} \|m\|_{S^{\vec{\eta}}_{q,q}B(\R^{\vec{n}})}\]
for $\vec{\eta} > \vec{\gamma}/2 + \vec{\theta}/2 + \vec{n}$. By interpolation, we obtain that
\[\|\Kern_{L^\times} m\|_{L^2(G^\times,\langle x_1 \rangle_{G_1}^{\gamma_1} \cdots \langle x_{\ell} \rangle_{G_\ell}^{\gamma_\ell} u_{\vec{\alpha}}^{-1}(x) \,dx)} \leq C_{K,\vec{\alpha},\vec{\gamma},\vec{\eta},q} \|m\|_{S^{\vec{\eta}}_{q,q}B(\R^{\vec{n}})}\]
for $\vec{\eta} > \vec{\gamma}/2 + \vec{n}/q - \vec{d}/\max\{2,q\}$.

In particular, if we take $\vec{\gamma} = (\gamma_{\vec{\alpha},1},\dots,\gamma_{\vec{\alpha},\ell})$, $\vec{\eta} = \vec{\alpha} + \vec{n}/q - \vec{d}/\max\{2,q\}$ and set $w_{\vec{\beta}} = w_{\vec{\beta},1} \otimes \dots \otimes w_{\vec{\beta},\ell}$, we get
\[\|\Kern_{L^\times} m\|_{L^2(G^\times,w_{\vec{\beta}}^{-1}(x) \,dx)} \leq C_{K,\vec{\beta},q} \|m\|_{S^{\vec{\eta}}_{q,q}B(\R^{\vec{n}})},\]
for every compact $K = \prod_{l=1}^\ell K_l \subseteq \prod_{l=1}^\ell (\R^{n_\ell} \setminus \{0\})$ and every $m \in \D(\R^{\vec{n}})$ with $\supp m \subseteq K$. The conclusion then follows by approximation.
\end{proof}

Notice that, in the particular case $\ell = 1$, the previous proposition, together with H\"older's inequality and Proposition~\ref{prp:interpolatedweightedestimates}, gives the following

\begin{cor}\label{cor:hypothesescomparison}
Suppose that a homogeneous weighted subcoercive system $L_1,\dots,L_n$ on a homogeneous Lie group $G$ satisfies \HPK{s}{d}. Then, for $q \in [1,\infty]$, it satisfies also \HP{q}{s + n/q - d/\max\{2,q\}}. In particular, $s \geq d/2$.
\end{cor}

The weighted estimate on $G^\times$ given by Proposition~\ref{prp:productestimates} are the starting point for the following multi-variate multiplier results. In fact, we are going to consider a setting which is more general than the product group $G^\times$.

Let $G$ be a connected Lie group, endowed with Lie group homomorphisms
\[\upsilon_l : G_l \to G \qquad\text{for $l=1,\dots,\ell$.}\]
Then, for $l=1,\dots,\ell$, the operators $L_{l,1},\dots,L_{l,n_l}$ correspond (via the derivative $\upsilon_l'$ of the homomorphism) to operators $L^\flat_{l,1}, \dots L^\flat_{l,n_l} \in \Diff(G)$, which are essentially self-adjoint. Since we want to give a meaning to joint functions of these operators on $G$, we suppose in the following that $L^\flat_{1,1}, \dots L^\flat_{1,n_1},\dots,L^\flat_{\ell,1}, \dots, L^\flat_{\ell,n_\ell}$ commute strongly, i.e., they admit a joint spectral resolution $E^\flat$ on $L^2(G)$.

In order to obtain multiplier results on $G$, we would like to ``transfer'' to $G$ the estimates obtained on the product group $G^\times$. However, we cannot apply directly the classical transference results (cf.\ \cite{coifman_transference_1976,berkson_transference_1996,cowling_herzs_1997}), since the map
\[\upsilon^\times : G^\times \ni (x_1,\dots,x_n) \mapsto \upsilon_1(x_1) \cdots \upsilon_\ell(x_\ell) \in G\]
in general is not a group homomorphism
and consequently it does not yield an action of $G^\times$ on $L^p(G)$ by translations. Nevertheless, under the sole assumption of (strong) commutativity of the differential operators $L^\flat_{l,j}$ on $G$, we are able to express the operator $m(L^\flat)$ on $G$ by a sort of convolution with the kernel $\Kern_{L^\times} m$ of the operator $m(L^\times)$ on $G^\times$.

\begin{prp}\label{prp:convolutionproduct}
(i) For every $m \in \D(\R^{\vec{n}})$ and $\phi \in L^2 \cap C_0(G)$,
\[m(L^\flat) \phi(x) = \int_{G^\times} \phi(x \,\upsilon^\times(y)^{-1} ) \, \Kern_{L^\times} m(y) \,dy.\]

(ii) Under the hypotheses of Proposition~\ref{prp:productestimates}, the previous identity holds for every $m \in S^{\vec{\beta}}_{q,q}B(\R^{\vec{n}})$ with compact support $\supp m \subseteq \prod_{l=1}^\ell (\R^{n_l} \setminus \{0\})$.
\end{prp}
\begin{proof}
(i) If $m \in \D(\R^{\vec{n}})$, then we can decompose $m = \sum_{k \in \N} g_{k,1} \otimes \cdots \otimes g_{k,\ell}$, where $g_{k,l} \in \D(\R^{n_l})$ for $k \in \N$, $l = 1,\dots,\ell$, and the convergence is in $\D(\R^{\vec{n}})$. In particular, by applying \cite[Corollary 5.5]{martini_spectral} and Theorem~\ref{thm:weightedestimates} to the group $G^\times$, we obtain that
\[\Kern_{L^\times} m = \sum_{k \in \N} \Kern_{L_1} g_{k,1} \otimes \cdots \otimes \Kern_{L_\ell} g_{k,\ell}\]
in $L^1(G^\times)$. On the other hand, for all $\phi \in L^2 \cap C_0(G)$, we have
\[g_{k,l}(L^\flat_l) \phi(x) = \int_{G_l} \phi(x \,\upsilon_l(y_l)^{-1}) 	\,\Kern_{L_l} g_{k,l}(y_l) \,dy_l\]
by \cite[Proposition 3.7]{martini_spectral}, and in particular (being $\Kern_{L_l} g_{k,l} \in L^1(G_l)$) also $g_{k,l}(L^\flat_l)\phi \in L^2 \cap C_0(G)$, so that, by iterating,
\[(g_{k,1} \otimes \dots \otimes g_{k,\ell})(L^\flat) \phi(x) = \int_{G^\times} \phi(x \,\upsilon^\times(y)^{-1}) \,\prod_{l=1}^\ell \Kern_{L_l} g_{k,l}(y_l) \,dy.\]
Summing over $k \in \N$, the left-hand side converges in $L^2(G)$ to $m(L^\flat)\phi$, whereas (since $y \mapsto \phi(x \,\upsilon^\times(y)^{-1})$ is bounded) the right-hand side converges pointwise to $\int_{G^\times} \phi(x \,\upsilon^\times(y)^{-1}) \, \Kern_{L^\times} m(y) \,dy$, and the conclusion follows.

(ii) Choose $\vec{\beta}'$ such that $\vec{\beta} > \vec{\beta}' > \vec{s} + \vec{n}/q - \vec{d}/\max\{2,q\}$. Take a compact $K = \prod_{l=1}^\ell K_l \subseteq \prod_{l=1}^\ell (\R^{n_l} \setminus \{0\})$ and a sequence $m_k \in \D(\R^{\vec{n}})$ with $\supp m_k \subseteq K$ such that $m_k \to m$ in $S^{\vec{\beta}'}_{q,q}(\R^{\vec{n}})$. By Proposition~\ref{prp:productestimates} and H\"older's inequality, we then have $\Kern_{L^\times} m_k \to \Kern_{L^\times} m$ in $L^1(G^\times)$; moreover, by Corollary~\ref{cor:hypothesescomparison}, $\beta'_l > n_l/q$ for $l=1,\dots,\ell$, so that $m_k \to m$ uniformly. Therefore the conclusion follows by applying (i) to the functions $m_k$ and passing to the limit.
\end{proof}

We are now going to exploit Littlewood-Paley theory. An important tool will be the following result, which summarizes a well-known argument for proving properties of square functions.

\begin{lem}\label{lem:khinchin}
Let $(X,\mu)$ be a $\sigma$-finite measure space, $1 \leq p < \infty$, $T_{\vec{k}}$ ($\vec{k} \in \N^\ell$) bounded linear operators on $L^p(X,\mu)$. Let $A > 0$ be such that, for all choices of $\varepsilon^i_k \in \{-1,1\}$ ($1 \leq i \leq \ell$, $k \in \N$) and of a finite subset $I \subseteq \N^\ell$, we have
\begin{equation}\label{eq:segni}
\left\|\textstyle\sum_{\vec{k} \in I} \varepsilon^1_{k_1} \cdots \varepsilon^\ell_{k_\ell} T_{\vec{k}}\right\|_{p \to p} \leq A.
\end{equation}
Then, for all $f \in L^p(X,\mu)$,
\begin{equation}\label{eq:khinchin1}
\left\|\left(\textstyle\sum_{\vec{k} \in \N^{\ell}} |T_{\vec{k}} f|^2\right)^{1/2}\right\|_p \leq C_{\ell,p} A \|f\|_p.
\end{equation}
Moreover, if $p > 1$, for all $\{f_{\vec{k}}\}_{\vec{k} \in \N^\ell} \subseteq L^p(X,\mu)$, if $\left(\sum_{\vec{k}} |f_{\vec{k}}|^2\right)^{1/2} \in L^p(X,\mu)$, then
\[\left\|\textstyle\sum_{\vec{k} \in \N^\ell} T_{\vec{k}} f_{\vec{k}}\right\|_p \leq C_{\ell,p'} A \left\|\left(\textstyle\sum_{\vec{k} \in \N^\ell} |f_{\vec{k}}|^2 \right)^{1/2} \right\|_p\]
where the series on the left-hand side converges unconditionally in $L^p$.
\end{lem}
\begin{proof}
For $n \in \N$, let $r_n : [0,1] \to \R$ be the $n$-th Rademacher function, $r_n(t) = (-1)^{\lfloor 2^n t \rfloor}$, and set $r_{\vec{k}} = r_{k_1} \otimes \cdots \otimes r_{k_\ell}$ for $\vec{k} \in \N^\ell$. Then $(r_{\vec{k}})_{\vec{k}}$ is an (incomplete) orthonormal system in $L^2([0,1]^\ell)$, and Khinchin's inequalities hold: for $1 \leq p < \infty$, there exist $c_{\ell,p},C_{\ell,p}> 0$ such that
\[c_{\ell,p}^{-1} \|f\|_p \leq \|f\|_2 \leq C_{\ell,p} \|f\|_p \qquad\text{for all $f \in \Span \{ r_{\vec{k}} \tc \vec{k} \in \N^\ell\}$.}\]
(see \cite[Appendix D]{stein_singular_1970} or \cite[Appendix C]{grafakos_classical_2008}).

Consequently, for all finite $I \subseteq \N^\ell$ and $f \in L^p(X,\mu)$, we have
\[\begin{split}
\left\|\left(\textstyle\sum_{\vec{k} \in I} |T_{\vec{k}} f|^2\right)^{1/2}\right\|_p^p &= \int_X \left(\textstyle\sum_{\vec{k} \in I} |T_{\vec{k}} f(x)|^2\right)^{p/2} \,d\mu(x)\\
&\leq C_{\ell,p}^p \int_X \int_{[0,1]^\ell} \left|\textstyle\sum_{\vec{k} \in I} T_{\vec{k}} f(x) r_{\vec{k}}(t) \right|^p \,dt \,d\mu(x) \\
&= C_{\ell,p}^p \int_0^1 \left\| \left(\textstyle\sum_{\vec{k} \in I} r_{\vec{k}}(t) T_{\vec{k}}\right) f \right\|^p \,dt \leq C_{\ell,p}^p A^p \|f\|_p^p.
\end{split}\]
Since $I \subseteq \N^\ell$ was arbitrary, \eqref{eq:khinchin1} follows by monotone convergence.

Notice now that the vector-valued Lebesgue space $V_p = L^p(X,\mu; l^2(\N^\ell))$ can be thought of as a space of sequences of $L^p(X,\mu)$-functions:
\[V_p = \left\{(f_{\vec{k}})_{\vec{k} \in \N^\ell} \in L^p(X,\mu)^{\N} \tc \textstyle \left(\sum_{\vec{k}} |f_{\vec{k}}|^2\right)^{1/2} \in L^p(X,\mu) \right\},\]
with norm $\|(f_{\vec{k}})_{\vec{k} \in \N^\ell}\|_{V_p} = \|\textstyle (\sum_{\vec{k}} |f_{\vec{k}}|^2)^{1/2}\|_p$. The inequality \eqref{eq:khinchin1} therefore means that the operator $f \mapsto (T_{\vec{k}} f)_{\vec{k} \in \N^\ell}$ is bounded $L^p(X,\mu) \to V_p$, with norm not greater than $C_{\ell,p} A$.

If $p > 1$, the hypothesis \eqref{eq:segni} is equivalent to
\[\left\|\textstyle\sum_{\vec{k} \in I} \varepsilon^1_{k_1} \cdots \varepsilon^n_{k_n} T_{\vec{k}}^*\right\|_{p' \to p'} \leq A;\]
consequently we have that $S : f \mapsto (T^*_{\vec{k}} f)_{\vec{k} \in \N^\ell}$ is bounded $L^{p'}(X,\mu) \to V_{p'}$, with norm not greater than $C_{\ell,p'} A$. This means that the transpose operator $S^* : V_{p} \to L^{p}(X,\mu)$ is bounded too, with the same norm; since
\[S^* \left( (f_{\vec{k}})_{\vec{k}} \right) = \textstyle\sum_{\vec{k}} T_{\vec{k}} f_{\vec{k}},\]
where the series on the right-hand side converges unconditionally in $L^p$, the remaining part of the conclusion follows.
\end{proof}

For $l=1,\dots,\ell$, let $\epsilon_{l,t}$ be the dilations on $\R^{n_l}$ associated to the weighted subcoercive system $L_{l,1},\dots,L_{l,n_l}$, and fix a $\epsilon_l$-homogeneous norm $|\cdot|_{\epsilon_l}$ on $\R^{n_l}$, smooth off the origin. Choose a non-negative $\xi \in \D(\R)$ with $\supp \xi \subseteq [1/2,2]$ and such that, if $\xi_k(t) = \xi(2^{-k} t)$, then
\begin{equation}\label{eq:sumofsquares}
{\textstyle\sum_{k \in \Z}} \xi_k^2(t) = 1 \qquad\text{for $t > 0$,}
\end{equation}
and set, for $l=1,\dots,\ell$ and $k \in \Z$, $\chi_{l,k}(\lambda) = \xi(|\epsilon_{l,2^{-k}}(\lambda)|_{\epsilon_l}) = \xi_k(|\lambda|_{\epsilon_l})$ for $\lambda \in \R^{n_l}$.
Moreover, for $\vec{k} = (k_1,\dots,k_\ell) \in \Z^\ell$, let $\chi_{\vec{k}} = \chi_{1,k_1} \otimes \cdots \otimes \chi_{\ell,k_\ell}$, $T_{\vec{k}} = \chi_{\vec{k}}(L^\flat)$, and define the square function
\[g(\phi) = \left( {\textstyle \sum_{\vec{k} \in \Z^\ell}} |T_{\vec{k}} \phi|^2 \right)^{1/2}.\]
Finally, set $X_{\vec{n}} = \{\lambda \in \R^{\vec{n}} \tc |\lambda_1|_{\epsilon_1} \cdots |\lambda_\ell|_{\epsilon_\ell} = 0\}$.

\begin{lem}\label{lem:littlewoodpaley}
For $1 < p < \infty$ and for all $\phi \in L^2 \cap L^p(G)$, 
\[c_p \|E^\flat(\R^{\vec{n}} \setminus X_{\vec{n}})\phi\|_p \leq \left\| g(\phi) \right\|_p \leq C_p \|\phi\|_p.\]
\end{lem}
\begin{proof}
Using the characterization \eqref{eq:mhpointwise} of $L^\infty$ Mihlin-H\"ormander conditions, it is not difficult to prove, for $l=1,\dots,\ell$, $s \in \N$, $(\varepsilon^l_k)_{k \in \Z} \in \{-1,0,1\}^\Z$, $N \in \N$, that
\[\left\|\textstyle \sum_{|k| \leq N} \varepsilon^l_k \chi_{l,k}\right\|_{M_{\epsilon_l} B^s_{\infty,\infty}} \leq C_{l,s},\]
where $C_{l,s} > 0$ does not depend on $(\varepsilon^l_k)_k$ or $N$.

By Theorem~\ref{thm:mihlinhoermander} applied to the group $G_l$, and by transference to the group $G$ (see \cite[Theorem 2.7]{berkson_transference_1996} and \cite[Proposition 3.7]{martini_spectral}), we then have
\[\left\|{\textstyle \sum_{|k| \leq N}} \varepsilon^l_k \chi_{l,k}(L^\flat_l) \right\|_{p \to p} \leq \left\|{\textstyle \sum_{|k| \leq N}} \varepsilon^l_k \chi_k(L_l) \right\|_{p \to p} \leq C_{l,p}\]
for $1 < p < \infty$, $l = 1,\dots,\ell$, where $C_{l,p} > 0$ does not depend on $(\varepsilon^l_k)_k$ or $N$, and consequently also
\[\left\|{\textstyle \sum_{|k_1|,\dots,|k_\ell| \leq N}} \varepsilon^1_{k_1} \cdots \varepsilon^\ell_{k_\ell} T_{\vec{k}} \right\|_{p \to p} \leq C_{1,p} \cdots C_{\ell,p}.\]
Moreover, by \eqref{eq:sumofsquares} and the properties of the spectral integral, $\sum_{\vec{k} \in \Z^\ell} T_{\vec{k}}^2$ converges strongly to $E^\flat(\R^{\vec{n}} \setminus X_{\vec{n}})$. The conclusion follows then immediately by Lemma~\ref{lem:khinchin}.
\end{proof}

In the following, we will consider Marcinkiewicz conditions on $\R^{\vec{n}}$ adapted to the system $\gimel_{\vec{t}} = \epsilon_{1,t_1} \times \dots \times \epsilon_{\ell,t_\ell}$ of multi-variate dilations.

\begin{theo}\label{thm:marcinkiewicz}
Suppose that, for $l=1,\dots,\ell$, the homogeneous group $G_l$, with the system $L_{l,1},\dots,L_{l,n_l}$, satisfies \HPK{s_l}{d_l}. If $q \in [1,\infty]$ and
\[\vec{\beta} > \vec{s} + \frac{\vec{n}}{q} - \frac{\vec{d}}{\max\{2,q\}},\]
then, for every Borel $m : \R^{\vec{n}} \to \C$ with $m|_{X_{\vec{n}}}= 0$ and $\|m\|_{M_\gimel S^{\vec{\beta}}_{q,q}B} < \infty$, the operator $m(L^\flat)$ is bounded on $L^p(G)$ for $1 < p < \infty$ and
\[\|m(L^\flat)\|_{p \to p} \leq C_{\vec{\beta},p,q} \|m\|_{M_\gimel S^{\vec{\beta}}_{q,q}B}.\]
\end{theo}
\begin{proof}
Choose a non-negative $\zeta \in \D(\R)$ with $\supp \zeta \subseteq [1/4,4]$ and such that $\zeta \equiv 1$ on $[1/2,2]$. For $l=1,\dots,\ell$, set $\eta_l(\lambda) = \zeta(|\lambda|_{\epsilon_l})$ and $\eta = \eta_1 \otimes \dots \otimes \eta_\ell$. If we set $m_{\vec{k}} = (m \circ \gimel_{(2^{k_1},\dots,2^{k_\ell})}) \eta$, $f_{\vec{k}} = m_{\vec{k}} \circ \gimel_{(2^{-k_1},\dots,2^{-k_\ell})}$ for $\vec k \in \Z^{\ell}$, then we have $\chi_{\vec{k}} m = f_{\vec{k}} \chi_{\vec{k}}$, so that $T_{\vec{k}} m(L^\flat) = f_{\vec{k}}(L^\flat) T_{\vec{k}}$.

Let $w_{\vec{\beta}} = w_{\vec{\beta},1} \otimes \dots \otimes w_{\vec{\beta},\ell} \in L^1(G^\times)$ be given by Proposition~\ref{prp:productestimates}. Set $w_{\vec{\beta},l,k} = 2^{-kQ_{\delta_l}} w_{\vec{\beta},l} \circ \delta_{l,2^{-k}}$ for $k \in \Z$, $l =1,\dots,\ell$, and let $w_{\vec{\beta},\vec{k}} = w_{\vec{\beta},1,k_1} \otimes \dots \otimes w_{\vec{\beta},\ell,k_\ell}$ for $\vec{k} \in \Z^\ell$. For $l=1,\dots,\ell$, if $\pi_l$ denotes the unitary representation of $G_l$ on $L^2(G)$ induced by the homomorphism $\upsilon_l$, since $w_{\vec{\beta},l}$ is M-admissible on $G_l$, then the maximal function $M_{\vec{\beta},l}$ on $G$ defined by $M_{\vec{\beta},l} \phi(x) = \sup_{k \in \Z} |\pi_l(w_{\vec{\beta},l,k}) \phi(x)|$ is bounded on $L^p(G)$ for $1 < p < \infty$, by transference \cite[Theorem 2.11]{berkson_transference_1996}.

If $\phi \in L^2 \cap C_0(G)$, then we have, by Proposition~\ref{prp:convolutionproduct}(ii) and H\"older's inequality,
\begin{multline*}
|f_{\vec{k}}(L^\flat) T_{\vec{k}} \phi(x)|^2 \leq \left(\int_{G^\times} |T_{\vec{k}} \phi(x \,\upsilon^\times(y)^{-1} )| |\Kern_{L^\times} f_{\vec{k}}(y)| \,dy\right)^2 \\
\leq \int_{G^\times} |T_{\vec{k}} \phi(x \,\upsilon^\times(y)^{-1} )|^2 w_{\vec{\beta},\vec{k}}(y) \,dy  \int_{G^\times} |\Kern_{L^\times} m_{\vec{k}}(y)|^2 w_{\vec{\beta}}^{-1}(y) \,dy\\
\leq C_{\vec{\beta},q} \|m_{\vec{k}}\|_{S^{\vec{\beta}}_{q,q}B(\R^{\vec{n}})}^2 \pi_1(w_{\vec{\beta},1,k_1}) \cdots \pi_\ell(w_{\vec{\beta},\ell,k_\ell}) (|T_{\vec{k}} \phi|^2)
\end{multline*}
thus
\begin{multline*}
\| g( m(L^\flat)\phi ) \|_p \\
\leq C_{\vec{\beta},q} \|m\|_{M_\gimel S^{\vec{\beta}}_{q,q}B} \left\| {\textstyle\sum_{\vec{k} \in \Z^\ell}} \pi_1(w_{\vec{\beta},1,k_1}) \cdots \pi_\ell(w_{\vec{\beta},\ell,k_\ell}) (|T_{\vec{k}} \phi|^2) \right\|^{1/2}_{p/2}
\end{multline*}
for $2 \leq p < \infty$.

On the other hand, since $w_{\vec{\beta}} = w_{\vec{\beta}}^*$, for every $\psi \in L^{(p/2)'}(G)$ we have
\[\begin{split}
&\left|\int_G \left({\textstyle\sum_{\vec{k} \in \Z^\ell}} \pi_1(w_{\vec{\beta},1,k_1}) \cdots \pi_\ell(w_{\vec{\beta},\ell,k_\ell}) (|T_{\vec{k}} \phi|^2) \right) \psi \,d\mu_G\right| \\
&\phantom{MMMM}\leq \sum_{\vec{k} \in \Z^\ell} \int_G \left( \pi_1(w_{\vec{\beta},1,k_1}) \cdots \pi_\ell(w_{\vec{\beta},\ell,k_\ell}) (|T_{\vec{k}} \phi|^2) \right) |\psi| \,d\mu_G\\
&\phantom{MMMM}\leq \int_G \left({\textstyle\sum_{\vec{k} \in \Z^\ell}} |T_{\vec{k}} \phi|^2 \right) M_{\vec{\beta},\ell} \cdots M_{\vec{\beta},1} (|\psi|) \,d\mu_G \\
&\phantom{MMMM}\leq C_{\vec{\beta},p} \left\|{\textstyle\sum_{\vec{k} \in \Z^\ell}} |T_{\vec{k}} \phi|^2 \right\|_{p/2} \|\psi\|_{(p/2)'},
\end{split}\]
that is,
\[
\left\| {\textstyle\sum_{\vec{k} \in \Z^\ell}} \pi_\ell(w_{\vec{\beta},\ell,k_\ell}) \cdots \pi_1(w_{\vec{\beta},1,k_1}) (|T_{\vec{k}} \phi|^2) \right\|_{p/2} \leq C_{\vec{\beta},p} \| g(\phi)\|_{p}^{2}.
\]
Putting all together, and applying Lemma~\ref{lem:littlewoodpaley}, we get the conclusion for $2 \leq p < \infty$ (notice that $E^\flat(\R^{\vec{n}} \setminus X_{\vec{n}}) m(L^\flat) = m(L^\flat)$ because $m|_{X_{\vec{n}}} = 0$). Thus we are done when $m(L^\flat)$ is self-adjoint, i.e., when $m$ is real-valued; in the general case, one can decompose $m$ in its real and imaginary parts and then apply the previous result to each part.
\end{proof}

The hypothesis $m|_{X_{\vec{n}}} = 0$ in Theorem~\ref{thm:marcinkiewicz} does not have an analogue in Theorem~\ref{thm:mihlinhoermander}, because the spectral measure of the origin for a homogeneous weighted subcoercive system on a homogeneous group is zero. On the other hand, if $E_l^\flat$ is the joint spectral resolution of $L_{l,1}^\flat,\dots,L_{l,n_l}^\flat$, then $E_l^\flat(\{0\})$ needs not be zero. However we have the following

\begin{prp}\label{prp:zeroprojection}
$E_l^\flat(\{0\})$ is bounded on $L^p(G)$ for $1 \leq p \leq \infty$. If moreover $\overline{\upsilon_l(G_l)}$ is not compact in $G$, then $E_l^\flat(\{0\}) = 0$.
\end{prp}
\begin{proof}
Choose $\psi \in \D(\R^{n_l})$ with $\psi(0) = 1$, so that $\psi_t = \psi \circ \epsilon_{l,t} \to \chr_{\{0\}}$ pointwise for $t \to +\infty$, and then $\psi_t(L_l^\flat) \to E_l^\flat(\{0\})$ strongly as operators on $L^2(G)$. By \cite[Proposition 3.7]{martini_spectral} we have $\psi_t(L_l^\flat) = \pi_l(\Kern_{L_l} \psi_t)$, thus
\[\|\psi_t(L_l^\flat)\|_{L^p(G)\to L^p(G)} \leq \|\Kern_{L_l} \psi_t\|_{L^1(G_l)} = \|\Kern_{L_l} \psi\|_{L^1(G_l)} < \infty\]
by \eqref{eq:homogeneity} and Theorem~\ref{thm:weightedestimates}. For every $f \in L^2 \cap L^p(G)$ and $g \in L^2 \cap L^{p'}(G)$, we then have
\[|\langle E_l^\flat(\{0\}) f, g \rangle| = \lim_{t \to +\infty} |\langle \psi_t(L_l^\flat) f, g \rangle| \leq \|\Kern_{L_l} \psi\|_{L^1(G_l)} \|f\|_{L^p(G)} \|g\|_{L^{p'}(G)},\]
which gives the required boundedness of $E^\flat_l(\{0\})$.

Suppose now that $\overline{\upsilon_l(G_l)}$ is not compact and that $E_l^\flat(\{0\})f = f$ for some $f \in L^2(G)$. This means that $d\pi_l(L_{l,1}) f = \dots = d\pi_l(L_{l,n_l})f = 0$, and proceeding analogously as in the proof of \cite[Theorem~5.2]{martini_spectral} one gets that $\pi_l(y) f = f$ for every $y \in G_l$. If $f \neq 0$, we can find a compact $K \subseteq G$ such that $\int_K |f(x)|^2 \,dx \neq 0$; on the other hand, since $\overline{\upsilon_l(G_l)}$ is not compact, it is easy to construct inductively a sequence $\{y_k\}_{k \in \N}$ in $G_l$ such that the sets $K \upsilon_l(y_k)$ for $k \in \N$ are pairwise disjoint, but then
\[\int_G |f(x)|^2 \,dx \geq \sum_{k \in \N} \int_{K \upsilon_l(y_k)} |f(x)|^2 \,dx = \sum_{k \in \N} \int_{K} |f(x)|^2 \,dx = \infty,\]
contradiction. Hence $f = 0$, and then $E_l^\flat(\{0\}) = 0$ by arbitrariness of $f$.
\end{proof}

Proposition~\ref{prp:zeroprojection} allows to relax the hypothesis $m|_{X_{\vec{n}}} = 0$ in Theorem~\ref{thm:marcinkiewicz}. Namely, for $I \subseteq \{1,\dots,\ell\}$, let $\vec{n}_I = (n_l)_{l \in I}$, so that $\R^{\vec{n}_I} = \prod_{l \in I} \R^{n_l}$; let moreover $\iota_I : \R^{\vec{n}_I} \to \R^{\vec{n}}$ be the canonical linear embedding, and define on $\R^{\vec{n}_I}$ the system of multi-variate dilations $\gimel_{I,(t_l)_{l \in I}}= \prod_{l \in I} \epsilon_{l,t_l}$. Then the decomposition
\[m(L^\flat) = \sum_{I \subseteq \{1,\dots,\ell\}} \prod_{l \notin I} E^\flat_l(\{0\}) \prod_{l \in I} E^\flat_l(\R^{n_l}\setminus \{0\}) (m \circ \iota_I)(L^\flat_I),\]
where $L_I^\flat = (L_{l,j}^\flat)_{l \in I, 1 \leq j \leq n_l}$, shows that the $L^p$-boundedness of $m(L^\flat)$ can be obtained by applying Theorem~\ref{thm:marcinkiewicz} to the subsystems $L_I^\flat$ of $L^\flat$:

\begin{cor}\label{cor:marcinkiewicz}
Suppose that, for $l=1,\dots,\ell$, the homogeneous group $G_l$, with the system $L_{l,1},\dots,L_{l,n_l}$, satisfies \HPK{s_l}{d_l}. If $q \in [1,\infty]$ and
\[\vec{\beta} > \vec{s} + \frac{\vec{n}}{q} - \frac{\vec{d}}{\max\{2,q\}},\]
then, for every Borel $m : \R^{\vec{n}} \to \C$ we have
\begin{equation}\label{eq:multimarcinkiewicz}
\|m(L^\flat)\|_{p \to p} \leq C_{\vec{\beta},p,q} \sum_{I \in \mathcal{I}} \|m \circ \iota_I\|_{M_{\gimel_I} S^{\vec{\beta_I}}_{q,q}B},
\end{equation}
where $\mathcal{I}$ is the set of the $I \subseteq \{1,\dots,\ell\}$ such that $\prod_{l \notin I} E^\flat_l(\{0\}) \neq 0$, and where $\|m \circ \iota_I\|_{M_{\gimel_I} S^{\vec{\beta_I}}_{q,q}B} = |m(0)|$ for $I = \emptyset$.
\end{cor}

In particular, if all the $\overline{\upsilon_l(G_l)}$ are not compact, then the hypothesis $m|_{X_{\vec{n}}} = 0$ in Theorem~\ref{thm:marcinkiewicz} can be dropped.

We conclude the section with a comparison of the Mihlin-H\"ormander and Marcin\-kiewicz conditions, which constitute the hypotheses of Theorems~\ref{thm:mihlinhoermander} and \ref{thm:marcinkiewicz} respectively: we obtain that, under suitable hypotheses on the orders of smoothness, a Marcinkiewicz condition is weaker than the corresponding Mihlin-H\"ormander condition.

\begin{prp}\label{prp:mihlinmarcinkiewicz}
If $q \in [1,\infty]$ and $\beta_l > \tilde Q_{\epsilon_l}/q$ for $l=1,\dots,\ell$, then
\[\|m\|_{M_\gimel S^{\vec{\beta}}_{q,q} B} \leq C_{q,\vec{\beta}} \|m\|_{M_\epsilon B_{q,q}^{\beta_1+\dots+\beta_\ell}},\]
where $\epsilon_t = \gimel_{(t,\dots,t)} = \epsilon_{1,t} \times \dots \times \epsilon_{\ell,t}$. 
\end{prp}
\begin{proof}
For $q < \infty$, in view of the characterization of $S^{\vec{\beta}}_{q,q} B(\R^{\vec{n}})$ as a tensor product of the $B_{q,q}^{\beta_l}(\R^{n_l})$ (cf.\ \cite[Theorem 2.2]{sickel_tensor_2009}), from Lemma~\ref{lem:smalldilations} we immediately get
\begin{equation}\label{eq:multismalldilations}
\sup_{\vec{t} > 0, \, |\vec{t}|_\infty \leq 1} \|(f \circ \gimel_{\vec{t}}) \,\eta\|_{S^{\vec{\beta}}_{q,q} B} \leq C_{\eta,q,\vec{\beta}} \|f\|_{S^{\vec{\beta}}_{q,q}B}
\end{equation}
for $\eta = \eta_1 \otimes \dots \otimes \eta_\ell \in \D(\R^{\vec{n}})$; the same holds also for $q = \infty$, as it is easily proved via the characterization by differences of the $S^{\vec{\beta}}_{\infty,\infty}B$-norm (cf.\ \cite[\S2.3.4]{schmeisser_topics_1987}).

Suppose now that $\supp \eta_l \subseteq \{ \lambda_l \in \R^{n_l} \tc a \leq |\lambda_l|_\infty \leq b\}$ for some $b > a > 0$ and $l=1,\dots,\ell$, and take $\tilde\eta \in \D(\R^{\vec{n}})$ such that $\tilde\eta|_{\{\lambda \tc a \leq |\lambda|_\infty \leq b\}} \equiv 1$. If $\vec{t} > 0$ and $|\vec{t}|_\infty = 1$, then $(\tilde \eta \circ \gimel_{\vec{t}}) \eta = \eta$, thus from \eqref{eq:multismalldilations} we get
\begin{equation}\label{eq:multifixeddilations}
\sup_{\vec{t} > 0, \, |\vec{t}|_\infty = 1} \|(f \circ \gimel_{\vec{t}}) \,\eta\|_{S^{\vec{\beta}}_{q,q} B} \leq C_{\eta,q,\vec{\beta}} \|f \tilde\eta\|_{S^{\vec{\beta}}_{q,q}B}.
\end{equation}

For an arbitrary $\vec{t} > 0$, set $r = |\vec{t}|_\infty$, so that $|r^{-1} \vec{t}|_\infty = 1$; then we have $m \circ \gimel_{\vec{t}} = (m \circ \epsilon_r) \circ \gimel_{r^{-1} \vec{t}}$, so that, by \eqref{eq:multifixeddilations} applied to $f = m \circ \epsilon_r$,
\[\|(m \circ \gimel_{\vec{t}}) \, \eta\|_{S^{\vec{\beta}}_{q,q} B} \leq C_{\eta,q,{\vec{\beta}}} \|(m \circ \epsilon_r) \, \tilde \eta\|_{S^{\vec{\beta}}_{q,q} B} \leq C_{\eta,q,{\vec{\beta}}} \|(m \circ \epsilon_r) \, \tilde \eta\|_{B^{\beta_1 + \dots + \beta_\ell}_{q,q}},\]
and the conclusion follows by a suitable choice of $\eta$ and $\tilde\eta$.
\end{proof}

Notice that a Mihlin-H\"ormander condition on $m: \R^{\vec{n}} \to \C$ gives some control also on the restriction of $m$ to $X_{\vec{n}} \setminus \{0\}$, so that it can be used to satisfy the more involved hypothesis of Corollary~\ref{cor:marcinkiewicz}. Namely, by the trace theorem for Besov spaces, under the hypothesis on $\beta_1,\dots,\beta_\ell$ of Proposition~\ref{prp:mihlinmarcinkiewicz}, if $m$ satisfies an $L^q$ Mihlin-H\"ormander condition of order $\beta_1 + \dots + \beta_\ell$, then $m \circ \iota_I$ satisfies an $L^q$ Mihlin-H\"ormander condition of order $\sum_{l \in I} \beta_l$ for $\emptyset \neq I \subseteq \{1,\dots,\ell\}$; therefore, by Proposition~\ref{prp:mihlinmarcinkiewicz}, all the summands in the right-hand side of \eqref{eq:multimarcinkiewicz}, except possibly for $|m(0)|$, are majorized by $\|m\|_{M_\epsilon B^{\beta_1 + \dots + \beta_\ell}_{q,q}}$.

\section{Examples and applications}

\subsection{Multipliers for a single operator}
Although the present work focuses on $L^p$ multipliers for systems of multiple operators, some results can be deduced also for single operators.

In view of the characterization stated in \S\ref{section:introduction}, a homogeneous weighted subcoercive system made of a single operator $L$ is simply a self-adjoint Rockland operator. Hence from Theorem~\ref{thm:mihlinhoermander} and Proposition~\ref{prp:HPsatisfaction} we get

\begin{cor}\label{cor:singleoperator}
Let $L$ be a self-adjoint Rockland operator on a homogeneous Lie group $G$. Suppose that $G$ is $h$-capacious, and let $Q_G$ be its dimension at infinity. If $m : \R \to \C$ satisfies an $L^2$ Mihlin-H\"ormander condition of order $s > (Q_G-h)/2$, then $m(L)$ is of weak type $(1,1)$ and bounded on $L^p(G)$ for $1 < p < \infty$.
\end{cor}

For a general (positive) Rockland operator, in the literature this result is stated either without specifying the regularity threshold \cite{hulanicki_functional_1984,stempak_weighted_1987} or for a particular class of groups \cite{hebisch_multiplier_1995}. By restricting to the case of a homogeneous sublaplacian $L$ on a stratified Lie group $G$, we recover the result of \cite{mauceri_vectorvalued_1990,christ_multipliers_1991}, where the threshold is half the homogeneous dimension of $G$.

Notice that, when a homogeneous Lie group $G$ is stratified (i.e., when the elements of degree $1$ generate the whole Lie algebra), then the homogeneous dimension coincides with $Q_G$. However, on a nilpotent Lie group $G$ there may be multiple homogeneous structures, and the homogeneous dimension $Q_\delta$ depends on the chosen automorphic dilations $\delta_t$ (for instance, if $X,Y,T$ is a basis of the Lie algebra of the Heisenberg group with $[X,Y]=T$, then we can set $\delta_t(X) = tX$, $\delta_t(Y) = t^a Y$, $\delta_t(T) = t^{1+a}T$ for all $a \geq 1$, and we have $Q_\delta = 2 + 2a$), whereas $Q_G$ is intrinsic of the Lie group structure of $G$. In fact (under the hypothesis that all the homogeneity degrees are not less than $1$) we always have $Q_\delta \geq Q_G$, with equality if and only if $G$ is stratified \cite[Proposition 2.2]{martini_spectral}.

The existence of a homogeneous sublaplacian forces $G$ to be stratified (modulo rescaling the homogeneity degrees). On the other hand, on non-stratified homogeneous groups $G$ there might exist higher-order self-adjoint Rockland operators $L$ (for instance, in the previous example of the Heisenberg group, one can take $L = (-iX)^{2a} -Y^2$ when $a \in \N$), to which Corollary~\ref{cor:singleoperator} applies, with threshold (at most) $Q_G/2$. Therefore our result is also an improvement of \cite[Corollary 7.1]{duong_plancherel-type_2002}, where the required threshold is $Q_\delta/2$.

Multiplier results for a single operator can be deduced from Theorem~\ref{thm:marcinkiewicz} too, through a sort of ``spectral mapping''; some examples in the context of a non-nilpotent Lie group are presented in the following \S\ref{subsection:nonnilpotent}.

\subsection{Plancherel measure and capacity map}\label{subsection:calculations}
In order to obtain the sharpest results from the previous multiplier theorems, properties of the Plancherel measure associated with a weighted subcoercive system and of the capacity map of a group must be investigated.

If $L_1,\dots,L_n$ is a weighted subcoercive system on a nilpotent group $G$, then the associated Plancherel measure $\sigma$ is related to the group Plancherel measure, defined on the set $\widehat G$ of (equivalence classes of) irreducible (unitary) representations of $G$. In fact, for every irreducible representation $\pi$ of $G$ on a Hilbert space $\HH_\pi$, we can find\footnote{If $\Delta = p(L_1,\dots,L_n)$ is a positive weighted subcoercive operator, then $\pi(\Kern_L(e^{-p}))$ is compact (since $G$ is CCR), thus $d\pi(\Delta)$ has discrete spectrum and finite-dimensional eigenspaces, and moreover it commutes with $d\pi(L_1),\dots,d\pi(L_n)$.} a complete orthonormal system $\{v_{\pi,\alpha}\}_\alpha$ of $\HH_\pi$ made of joint eigenvectors of $d\pi(L_1),\dots,d\pi(L_n)$. If $\lambda_{\pi,\alpha} \in \R^n$ denotes the eigenvalues of $d\pi(L_1),\dots,d\pi(L_n)$ corresponding to the eigenvector $v_{\pi,\alpha}$, then, for every $m \in \D(\R^n)$,
\[
\int_{\R^n} |m(\lambda)|^2 \,d\sigma(\lambda) = \|\breve m\|^2_{L^2(G)} = \int_{\widehat G} \|\pi(\breve m)\|_{\HS}^2 \,d\pi = \int_{\widehat G} \sum_\alpha |m(\lambda_{\pi,\alpha})|^2 \,d\pi
\]
(cf.\ \cite[Proposition 3.7]{martini_spectral}). If one is able to determine both the group Plancherel measure and the eigenvectors $v_{\pi,\alpha}$ in such a way that the function $(\pi,\alpha) \mapsto \lambda_{\pi,\alpha}$ is sufficiently regular, then the measure $\sigma$ on $\R^n$ is determined by the previous identity as the push-forward of the product of the group Plancherel measure times a counting measure.


This route can be followed, e.g., for the free $2$-step nilpotent group on $3$ generators $N_{3,2}$, which is determined by the relations
\[[X_1,X_2] = T_3, \qquad [X_2,X_3] = T_1, \qquad [X_3,X_1]=T_2,\]
where $X_1,X_2,X_3,T_1,T_2,T_3$ is a basis of its Lie algebra $\lie{n}_{3,2}$; the dilations $\delta_t$ given by $\delta_t(X_j) = tX^j$, $\delta_t(T_j) = t^2 T_j$ define a stratification of $N_{3,2}$, so that $Q_{N_{3,2}} = Q_\delta = 9$.

If $L = -(X_1^2 + X_2^2+X_3^2)$ is the sublaplacian, $\Delta = -(T_1^2 +T_2^2+T_3^2)$ is the central Laplacian and $D = -(X_1 T_1 + X_2T_2 + X_3T_3)$, then $L,\Delta,D$ is a homogeneous weighted subcoercive system, with Plancherel measure $\sigma$ given by
\[\int_{\R^3} f \,d\sigma \\
= \sum_{\alpha \in 2\N+1} \int_0^\infty \int_{-1}^1 f\left(\lambda,\lambda^{3/2}{\textstyle\frac{\theta(1-\theta^2)}{\alpha}},\lambda^2 {\textstyle\frac{(1-\theta^2)^2}{\alpha}}\right) \,{\textstyle\frac{(1-\theta^2)^3}{8\pi^4\alpha^4}} \,d\theta \,\lambda^{7/2} \,d\lambda
\]
(cf.\ \cite{fischer_gelfand_2009}). This measure is supported on a countable family of surfaces accumulating on the axis $\R \times \{0\} \times \{0\}$, and it is not difficult to show that $\sigma$ is locally $2$-bounded on $\R^3 \setminus \{0\}$, so that the system $L,\Delta,D$ satisfies the hypothesis \HPK{9/2}{2}.

On the group $N_{3,2}$ we can also consider the system $L,-iT_1,-iT_2,-iT_3$; in this case the Plancherel measure $\sigma$ is given by
\[\int_{\R^4} f \,d\sigma = \sum_{\alpha \in 2\N+1} \int_0^\infty \int_{S^2} \int_0^1 f\left(\lambda,\lambda {\textstyle\frac{\nu\omega}{\alpha}}\right) \,{\textstyle\frac{\nu^3}{32 \pi^5 \alpha^4 \sqrt{1-\nu}}} \,d\nu \,d\omega \,\lambda^{7/2} \,d\lambda.\]
One can show that $\sigma$ is locally $\frac{7}{2}$-bounded on $\R^4 \setminus \{0\}$, so that the system $L,-iT_1,-iT_2,-iT_3$ satisfies \HPK{9/2}{7/2}.

As an example of computation of the capacity map $J$, we consider instead the $3$-step group $G_{6,23}$ of \cite{nielsen_unitary_1983}, which is defined by the relations
\[\begin{split}
&[X_6,X_5]= X_4, \qquad [X_6,X_4]=X_2, \qquad [X_6,X_3]=-X_1, \\
&[X_5,X_4]=X_1, \qquad [X_5,X_3]=X_2,
\end{split}\]
where $X_6,X_5,X_4,X_3,X_2,X_1$ is a basis of its Lie algebra $\lie{g} = \lie{g}_{6,23}$. It can be shown that the unique automorphic dilations $\delta_t$ (modulo automorphisms and rescaling) are given by $\delta_t(X_j) = t^{w_j} X_j$ with $w_6=w_5 = 1$, $w_4=w_3=2$, $w_2=w_1=3$, so that this is an example of a non-stratifiable group, with $12 = Q_\delta > Q_{G_{3,2}} = 11$.

With the notation of \S\ref{section:metivier}, we have 
\[\lie{z} = \Span\{X_2,X_1\}, \qquad \lie{y} = \Span\{X_4,X_3,X_2,X_1\}.\]
If we denote by $\bar X_6,\bar X_5,\bar X_4,\bar X_3$ and $X_2^*, X_1^*$ the bases induced by $X_6,\dots,X_1$ on $\lie{g}/\lie{z}$ and $\lie{z}^*$ respectively, then we have
\[|J(x_6 \bar X_6 + x_5 \bar X_5 + x_4 \bar X_4 + x_3 \bar X_3, t_2 X_2^* + t_1 X_1^*)|^2 = (x_6^2+x_5^2)(t_2^2+t_1^2)\]
with respect to a suitable norm on $\lie{y}^*$, therefore the dual elements $\bar X_6^*,\bar X_5^* \in (\lie{g}/\lie{z})^*$ and $X_2,X_1 \in \lie{z}$ attest that $G_{6,23}$ is 2-capacious (despite the fact that Proposition~\ref{prp:capacitycriteria} does not apply to this group). Consequently, every homogeneous weighted subcoercive system on $G_{6,23}$, such as
\[(-iX_6)^{4k}+(-iX_5)^{4k}+(-iX_3)^{2k}, \quad -iX_2, \quad -iX_1\]
for $k \in \N \setminus \{0\}$, satisfies the hypothesis \HPK{9/2}{1}.

Further examples and details may be found in \cite{martini_multipliers_2010}.

\subsection{Non-nilpotent groups}\label{subsection:nonnilpotent}

Theorem~\ref{thm:marcinkiewicz} allows one to obtain multiplier theorems also on groups which are not homogeneous, even not nilpotent. An interesting class of examples comes by considering an action of a torus $\T^d = \R^d/\Z^d$ on a homogeneous group $N$ by automorphisms which commute with dilations, and the corresponding semidirect product $N \rtimes \T^d$ (or alternatively its universal covering group $N \rtimes \R^d$).

Take for instance a diamond group $G = H_n \rtimes \T^d$ (see \cite{ludwig_dual_1995}). If $L$ is a $\T^d$-invariant homogeneous sublaplacian on $H_n$ and $U_1,\dots,U_d$ are the partial derivatives on the torus $\T^d$, then $L,-iU_1,\dots,-iU_d$ is a weighted subcoercive system on $G$, since these operators commute and they generate an algebra containing the sublaplacian $\Delta = L + (-iU_1)^2 + \dots + (-iU_d)^2$. In fact, each of the operators $L,-iU_1,\dots,-iU_d$ can be considered as a homogeneous weighted subcoercive system in itself: $L$ is Rockland on $H_n$, and therefore satisfies \HPK{(\dim H_n)/2}{1}, whereas $-iU_j$ comes from the corresponding derivative on the $j$-th factor of $\R^d$, which satisfies \HPK{1/2}{1}. By applying Theorem~\ref{thm:marcinkiewicz}, we then obtain that, if $m : \R^{1+d} \to \C$ vanishes on the coordinate hyperplanes and
\[\|m\|_{M_\gimel S^{\vec{s}}_{2,2}B(\R^{1+d})} < \infty \qquad\text{for } \vec{s} > \left(\frac{\dim H_n}{2},\frac{1}{2},\dots,\frac{1}{2}\right),\]
then $m(L,-iU_1,\dots,-iU_d)$ is bounded on $L^p(G)$ for $1 < p < \infty$.

Thanks to Corollary~\ref{cor:marcinkiewicz} and Proposition~\ref{prp:mihlinmarcinkiewicz}, this result in turn yields a multiplier theorem for the sublaplacian $\Delta$: if $m : \R \to \C$ satisfies
\[\|m\|_{M_\epsilon B^{s}_{2,2}(\R)} < \infty \qquad\text{for } s > \frac{\dim H_n + d}{2} = \frac{\dim G}{2},\]
then $m(\Delta)$ is bounded on $L^p(G)$ for $1 < p < \infty$. We remark that:
\begin{itemize}
\item this condition is sharper than the one following by the general result of Alexopoulos \cite{alexopoulos_spectral_1994}, which instead requires an $L^\infty$ condition of order $s > \frac{\dim G + 1}{2}$;
\item this is an example of a group in which the regularity threshold in a multiplier theorem can be lowered to half the topological dimension, which is neither a Heisenberg or related group, nor $SU_2$;
\item the sublaplacian $\Delta$ can be replaced by any operator of the form
\[L^{k_0} + -(iU_1)^{2k_1} + \dots + (-iU_d)^{2k_d} \qquad\text{or}\qquad L^{k_0} (-iU_1)^{k_1} \cdots (-iU_d)^{k_d}\]
for some $k_0,k_1,\dots,k_d \in \N \setminus \{0\}$, obtaining an analogous multiplier result with identical smoothness requirement.
\end{itemize}

Spectral multipliers for operators such as the complete Laplacian
\[\Delta_c = L + (-iT)^2 + (-iU_1)^2 + \dots + (-iU_d)^2,\]
where $T$ is the central derivative on $H_n$, can also be studied. By considering $L,-iT$ together as a homogeneous system on $H_n$, and each of the $-iU_j$ separately as before, one obtains a multiplier theorem for $\Delta_c$, with an $L^\infty$ condition of order $s > \frac{\dim G}{2}$.

Analogous considerations hold if one replaces $H_n$ by any M\'etivier group, and also if one takes the universal covering group $H_n \rtimes \R^d$; this last case comprises, for $d=1$, the oscillator groups. Notice that the previous result about the Laplacian $\Delta_c$, when stated on the universal covering group, is sharper than \cite{alexopoulos_spectral_1994}, since the degree of growth of the group is greater than its topological dimension.

Further examples include the plane motion group $\R^2 \rtimes \T$, and the semidirect product $N_{2,3} \rtimes \T$ determined by the action of $SO_2$ on the free $3$-step nilpotent group $N_{2,3}$ considered, e.g., in \cite[\S5.3]{martini_spectral}. In these last cases, for some distinguished sublaplacians, we still get a sharpening of the result by Alexopoulos: although the required order of smoothness is the same, our condition is expressed in terms of an $L^2$ instead of an $L^\infty$ Besov norm.

\section*{Acknowledgements}
I thank Fulvio Ricci for drawing my attention to the subject of this work, and for his continuous encouragement and support.

\bibliographystyle{abbrv}
\bibliography{../../multipliers}

\end{document}